%% file: domination.tex
\documentclass[12pt, a4paper,leqno,oneside]{amschanged}

\usepackage{amsmath}
\usepackage{amssymb}
\usepackage{amsthm}
\usepackage[foot]{amsaddr}

\usepackage{atbeginend}

\usepackage{psfrag}
\usepackage{graphicx}
\usepackage{hyperref}
\hypersetup{
    colorlinks,%
    citecolor=black,%
    filecolor=black,%
    linkcolor=black,%
    urlcolor=black
}

\usepackage[notref,notcite,color,final]{showkeys}
\usepackage{fullpage}

\def\ff#1{\mathsf #1}

\usepackage{type1cm}

\newtheorem{theorem}{Theorem}[section]

\newtheorem{definition}[theorem]{Definition}
\newtheorem{proposition}[theorem]{Proposition}
\newtheorem{lemma}[theorem]{Lemma}

\numberwithin{equation}{section}

\theoremstyle{remark}
\newtheorem{remark}[theorem]{Remark}

\newcommand{\capacity}{\, \textup{cap}}
\newcommand{\reg}{{t_*}}
\newcommand{\coup}{{Q}}
\newcommand{\rbd}{N^{-c_\epsilon}}
\newcommand{\rbdd}{N^{-c_\epsilon}}
\newcommand{\Rbd}{N^{-c_\epsilon}}

\def\supp{\mathop{\rm supp}\nolimits}

\def\diam{\mathop{\rm diam}\nolimits}

\def\ran{\mathop{\rm range}\nolimits}

\usepackage{titlesec}
\titleformat{\section}{\Large\bfseries}{\thesection}{1em}{}
\titleformat{\subsection}{\bfseries}{\thesubsection}{1em}{}

\renewcommand{\theenumi}{\arabic{enumi}}
\renewcommand{\labelenumi}{\theenumi)}

\definecolor{orange}{rgb}{1,0.5,0}
\usepackage{courier}

\begin{document}

\fontsize{12}{14}\rm
\addtolength{\abovedisplayskip}{.5mm}
\addtolength{\belowdisplayskip}{.5mm}
\AfterBegin{enumerate}{\addtolength{\itemsep}{2mm}}


\title{\LARGE \usefont{T1}{tnr}{b}{n} \selectfont O\lowercase{n the fragmentation of a torus by random walk}} 


\author{\normalsize \itshape A.T\lowercase{eixeira} $^1$ \color{white} \tiny}
\address{$^1$ETH Zurich, Department of Mathematics, \newline R\"amistrasse 101, 8092 Zurich, Switzerland, \itshape{\texttt{augusto.teixeira@math.ethz.ch}}.
}
\author{\color{black} \normalsize \itshape D.W\lowercase{indisch} $^2$}
\address{$^2$ The Weizmann Institute of Science, Faculty of Mathematics and Computer Science,\newline Rehovot 76100, Israel, \itshape{\texttt{david.windisch@weizmann.ac.il}}.}


  \date{\today}

\begin{abstract}
We consider a simple random walk on a discrete torus $({\mathbb Z}/N{\mathbb Z})^d$ with dimension $d \geq 3$ and large side length $N$. For a fixed constant $u \geq 0$, we study the percolative properties of the vacant set, consisting of the set of vertices not visited by the random walk in its first $[uN^d]$ steps. We prove the existence of two distinct phases of the vacant set in the following sense: if $u>0$ is chosen large enough, all components of the vacant set contain no more than $(\log N)^{\lambda(u)}$ vertices with high probability as $N$ tends to infinity. On the other hand, for small $u>0$, there exists a macroscopic component of the vacant set occupying a non degenerate fraction of the total volume $N^d$. In dimensions $d \geq 5$, we additionally prove that this macroscopic component is unique by showing that all other components have volumes of order at most $(\log N)^{\lambda(u)}$. Our results thus solve open problems posed by Benjamini and Sznitman \cite{BS08}, who studied the small $u$ regime in high dimension. The proofs are based on a coupling of the random walk with random interlacements on ${\mathbb Z}^d$. 
Among other techniques, the construction of this coupling employs a refined use of discrete potential theory. By itself, this coupling strengthens a result in \cite{W08}.

\end{abstract}

\maketitle

\input{introduction}
\input{notation}

\input{preliminaries}

\input{quasi}

\input{poisson}
\input{above}
\input{below}

\input{torus}

\bibliographystyle{plain}
\bibliography{domination}


\end{document}

%% file: introduction.tex
\section{Introduction}
\label{s:intro}

We consider a simple random walk on the $d$-dimensional torus ${\mathbb T}_N = ({\mathbb Z}/N{\mathbb Z})^d$ with large side length $N$ and fixed dimension $d \geq 3$. The aim of this work is to improve our understanding of the percolative properties of the set of vertices not visited by the random walk until time $uN^d$, where the parameter $u>0$ remains fixed and $N$ tends to infinity. We refer to this set as the vacant set. The vacant set occupies a proportion of vertices bounded away from $0$ and $1$ as $N$ tends to infinity, so it is natural to study the sizes of its components. At this point, the main results on the vacant set are the ones of Benjamini and Sznitman \cite{BS08}, showing that for high dimensions $d$ and small parameters $u>0$, there is a component of the vacant set with cardinality of order $N^d$ with high probability. As is pointed out in \cite{BS08}, this result raises several questions, such as:
\begin{enumerate}
\item
Do similar results hold for any dimension $d \geq 3$?
\item
For small parameters $u>0$, does the second largest component have a volume of order less than $N^d$?
\item
Provided $u>0$ is chosen large enough, do all components of the vacant set have volumes of order less than $N^d$?
\end{enumerate}
The results of this work in particular give positive answers to these questions, and thereby confirm observations made in computer simulations (see Figure~\ref{fig:giant}). We thus prove the existence of distinct regimes for the vacant set as $u$ varies, similar to the ones exhibited by Bernoulli percolation on the torus and other random graph models. 
\begin{figure}[h]
\begin{center}
\includegraphics[angle=0, width=0.395\textwidth]{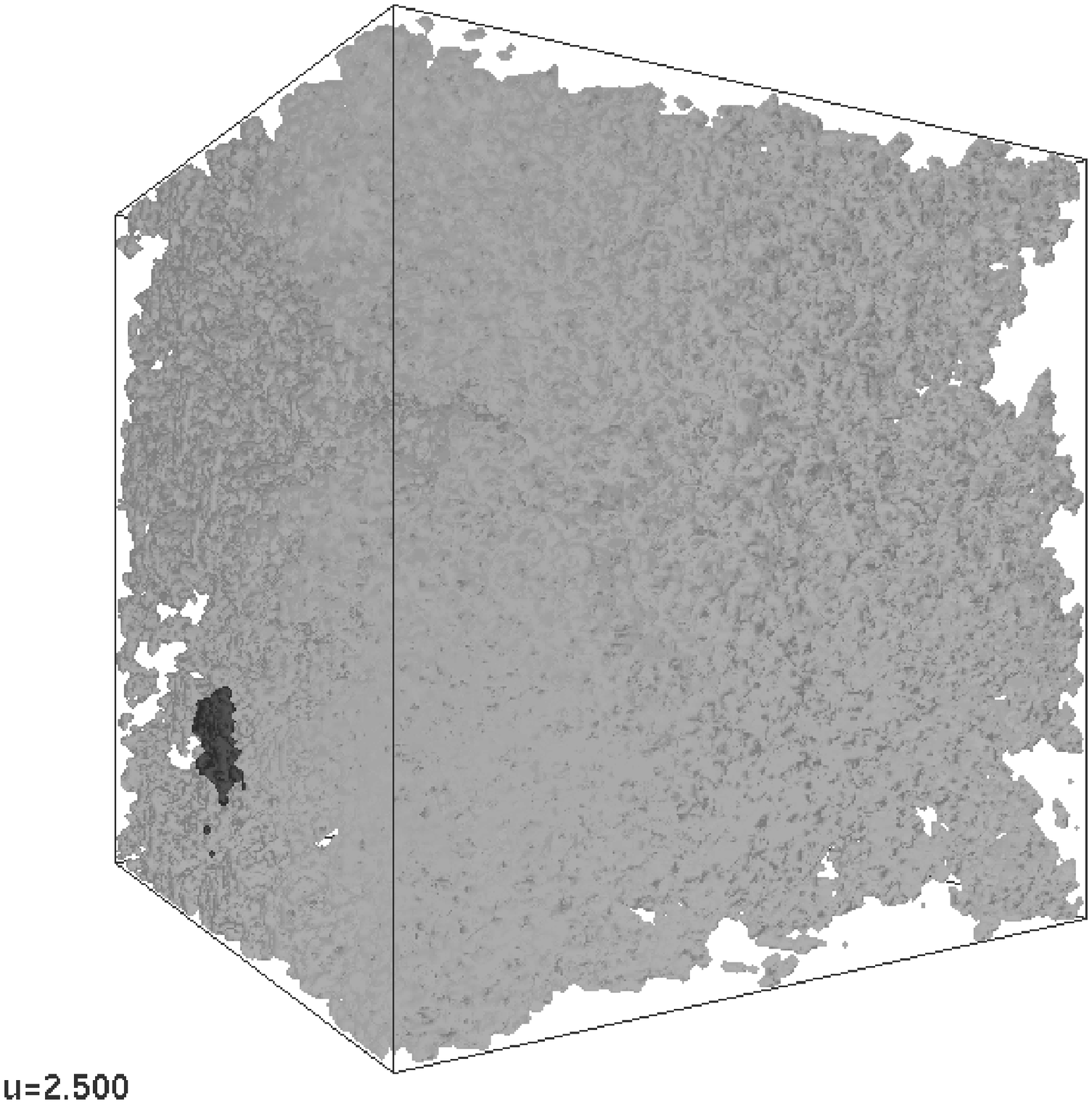}
\includegraphics[angle=0, width=0.395\textwidth]{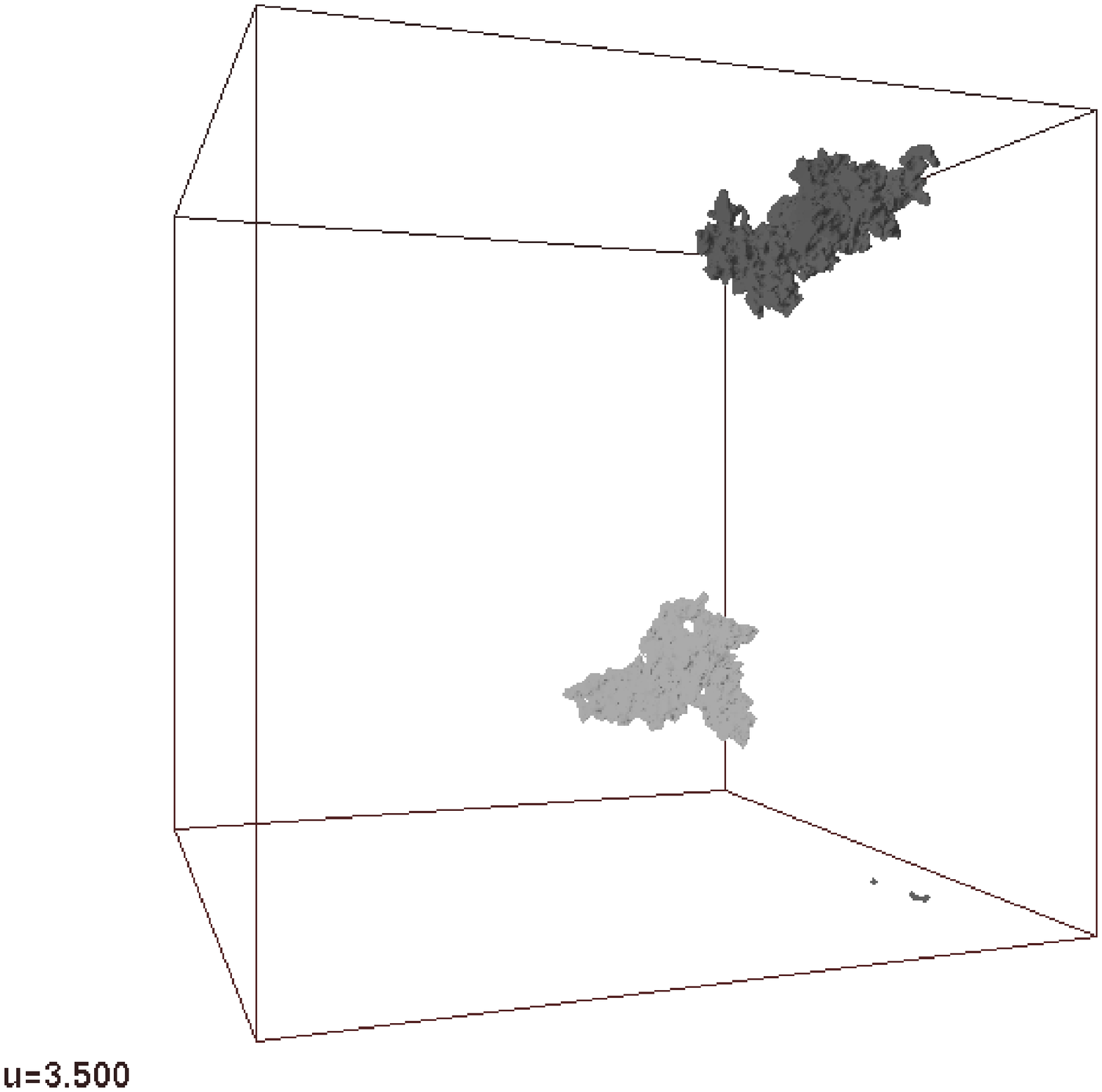}
\end{center}
\caption{A computer simulation of the largest component (light gray) and second largest component (dark gray) of the vacant set left by a random walk on $({\mathbb Z}/N{\mathbb Z})^3$ after $[uN^3]$ steps, for $N=200$. The picture on the left-hand side corresponds to $u=2.5$, the right-hand side to $u=3.5$. For more pictures, see \tiny\url{http://www.wisdom.weizmann.ac.il/\textasciitilde
davidw/torus.html}.} \label{fig:giant}
\end{figure}

Our answers are closely linked to Sznitman's model of random interlacements (cf.~\cite{Szn09}), which we now briefly introduce. The random interlacement ${\mathcal I}^u \subseteq {\mathbb Z}^d$ at level $u \geq 0$ is the trace left on ${\mathbb Z}^d$ by a cloud of paths constituting a Poisson point process on the space of doubly infinite trajectories modulo time-shift, tending to infinity at positive and negative infinite times. The parameter $u$ is a multiplicative factor of the intensity measure of this point process. In Section~\ref{s:prel} we give an explicit construction of the random interlacements process inside a box, see \eqref{e:interlac}. For now, let us just mention that the law $Q^u$ of $\mathcal{I}^u$ (regarded as a random subset of $\mathbb{Z}^d$) is characterized by the following equation:
\begin{align}
\label{e:char}
 Q^u \left[ {\mathcal I}^u \cap {\mathsf V} = \emptyset \right] = e^{-u \capacity ({\mathsf V})}, \text{ for all finite sets } {\mathsf V} \subset {\mathbb Z}^d,
\end{align}
where $\capacity ({\mathsf V})$ denotes the capacity of $\mathsf V$, defined in \eqref{d:cap} below, see (2.16) in \cite{Szn09}. The random interlacement describes the structure of the random walk trajectory on $\mathbb T$ in local neighborhoods. Indeed, for a fixed $\epsilon \in (0,1)$, consider the closed ball $A=B(0,N^{1-\epsilon}) \subset {\mathbb T}$ of radius $N^{1-\epsilon}$ centered at $0 \in {\mathbb T}$ with respect to the $\ell_\infty$-distance. Then $A$ is isomorphic to the ball ${\mathsf A} = B(0,N^{1-\epsilon}) \subset {\mathbb Z}^d$ via a graph isomorphism $\phi$, so we can consider the random subset of ${\ff A} \subseteq \mathbb{Z}^d$,
\begin{align}
\label{e:XuA}
X(u,{\mathsf A}) = \phi (X_{[0,uN^d]} \cap A),
\end{align}
where $X_{[0,uN^d]}$ is the random set of vertices visited in the first $[uN^d]$ steps of a simple random walk on $\mathbb T$ with uniformly distributed starting point. The following theorem shows that $X(u, {\mathsf A})$ can be approximated by random interlacements in a strong sense:

\medskip

\begin{theorem} \label{t:dom}
\textup{($d \geq 3$)}
For any $u>0$, $\alpha>0$, $\epsilon \in (0,1)$, there exists a constant $c$ depending on $d, u, \alpha, \epsilon$ and a coupling $(\Omega, {\mathcal A}, \coup)$ of $X_{[0,uN^d]}$ with random interlacements ${\mathcal I}^{u(1-\epsilon)}$ and ${\mathcal I}^{u(1+\epsilon)}$ on ${\mathbb Z}^d$, such that
\begin{align}
\label{e:dom} \coup \left[ {\mathcal I}^{u(1-\epsilon)} \cap {\mathsf A} \subseteq X(u,{\mathsf A}) \subseteq {\mathcal I}^{u(1+\epsilon)} \cap {\mathsf A} \right] \geq 1 - cN^{-\alpha}, \text{ for } N \geq 1.
\end{align}
\end{theorem}

\medskip

The above theorem indicates that percolative properties of the vacant set left by the random walk on $\mathbb{T}$ should be related to percolative properties of the vacant set 
\begin{align}
\label{e:Vu}
{\mathcal V}^u = {\mathbb Z}^d \setminus {\mathcal I}^u
\end{align}
left by the random interlacement. Indeed, our main theorems are applications of Theorem~\ref{t:dom} and results on random interlacements, some of which we now describe. It is known that ${\mathcal V}^u$ undergoes a phase transition at a critical threshold $u_\star \in (0, \infty)$, given by
\begin{align}
\label{e:ustar}
 u_\star = \inf \{ u \geq 0: \eta(u) = 0\},
\end{align}
where $\eta(u)$ is the percolation function
\begin{align}
\label{e:eta(u)}
 \eta(u) = Q^u \left[ 0 \text{ belongs to an infinite component of } {\mathcal V}^u \right], \, u \geq 0.
\end{align}
It is proved by Sznitman in \cite{Szn09} and by Sidoravicius and Sznitman in \cite{SS10} that indeed 
\begin{align*}
u_\star \in (0,\infty), \text{ for all } d \geq 3.
\end{align*}
Moreover, it is known that for $u>u_\star$, ${\mathcal V}^u$ consists of finite components, whereas for $u<u_\star$, ${\mathcal V}^u$ has a unique infinite component with probability $1$, see \cite{Szn09}, \cite{T08}.

For values of $u$ above another critical threshold $u_{\star \star} \geq u_\star$, the connectivity function of ${\mathcal V}^u$ is known to decay fast, see Theorem~0.1 of \cite{SS10}. For the precise definition of $u_{\star \star}$, we refer to \eqref{e:udouble} below. For now, let us just point out that
\begin{equation}
 \text{$u_{\star \star} < \infty$ for every $d \geq 3$,}
\end{equation}
and that it is an open problem whether $u_{\star \star}$ actually coincides with $u_{\star}$. We denote by ${\mathcal C}_{max}^u$ a connected component of $\mathbb{T} \setminus X_{[0,uN^d]}$ with largest volume and in the following result establish the existence of a large $u$ regime in which the vacant set consists of small components. This answers a question posed in \cite{BS08}, see the paragraph below (0.8).

\medskip

\begin{theorem}\label{t:torus}
\textup{($d \geq 3$)}
For all $u > u_\star$ and any $\eta >0$,
\begin{align}
\label{e:torus1} \lim_{N \to \infty} P[|{\mathcal C}_{max}^u| \geq \eta N^d ] =0,
\end{align}
and for all $u > u_{\star \star}$, there exists a $\lambda(u) > 0$ such that for any $\rho > 0$,
\begin{align}
\label{e:torus2} \lim_{N \to \infty} N^\rho P[|{\mathcal C}_{max}^u| \geq \log^\lambda N] =0.
\end{align}
\end{theorem}

\medskip

As another application of Theorem~\ref{t:dom}, we prove the existence of a small $u$ regime with a macroscopic component of the vacant set for all dimensions $d \geq 3$, thereby extending the main result of \cite{BS08} to lower dimensions.

\medskip

\begin{theorem}
\label{t:d3} 
\textup{($d \geq 3$)}
For $\epsilon, u > 0$ chosen small enough,
\begin{equation}
\lim_{N \to \infty} P[|\mathcal{C}^u_{max}| > \epsilon N^d]= 1.
\end{equation}
\end{theorem}

\medskip

We can strengthen the last theorem for so-called strongly supercritical parameters $u>0$. This notion is defined via geometric properties of ${\mathcal I}^u$ and is made precise in Definition~\ref{d:strong} below. For the moment, let us mention that
\begin{equation}
\label{e:strong4}
 \text{for $d \geq 5$, there exists a ${\bar u}_d>0$, such that all $u < {\bar u}_d$ are strongly supercritical,}
\end{equation}
see Theorems~3.2 and 3.3 in \cite{T09b}. It is an open problem whether in fact all parameters $u<u_\star$ are strongly supercritical for every $d \geq 3$, see also Remark~\ref{r:strong} below. We denote by ${\mathcal C}^u_{sec}$ the second largest component of the vacant set left by the walk. Or more precisely, to avoid ties we let ${\mathcal C}^u_{sec}$ be a component of ${\mathbb T} \setminus (X_{[0,uN^d]} \cup {\mathcal C}^u_{max})$ with largest volume.

\medskip

\begin{theorem}
\label{t:super}
If $u$ is strongly supercritical (cf.~\eqref{e:strong4}, Definition~\ref{d:strong}), then for $\eta(u)$ defined in \eqref{e:eta(u)} and every $\epsilon > 0$,
\begin{equation}
\label{e:density}
\lim_{N \to \infty} P\Big[ \Big|\frac{|\mathcal{C}^u_{max}|}{N^d} - \eta(u) \Big| > \epsilon \Big] = 0.
\end{equation}
Moreover, for $u$ strongly supercritical, there is a $\lambda = \lambda(u) > 0$ such that for every $\rho > 0$,
\begin{equation}
\label{e:sec}
\lim_{N \to \infty} N^\rho P [|\mathcal{C}^u_{sec}| > \log^\lambda N] =0.
\end{equation}
\end{theorem}

\medskip

The above theorems give strong answers to questions 1-3 mentioned in the beginning of this section, and hence solve some open problems mentioned in \cite{BS08} (see Remark~4.7(1) and the introduction). Theorem~\ref{t:dom} also strengthens the result of \cite{W08}, where weak convergence of random walk trajectories to random interlacements is shown for microscopic neighborhoods only. Some of the auxiliary estimates on expected entrance times, hitting distributions and the quasistationary distribution we develop in the proof of Theorem~\ref{t:dom} could also be of use in other contexts, see Proposition~\ref{p:Gloc}, and Lemmas~\ref{l:quasi}, \ref{l:quni}.

\medskip

Results similar to the above are proved in the recent work \cite{CTW} for random walks on random regular graphs with the help of random interlacements on regular trees, as well as in the recent work
\cite{CF10} using different methods. In \cite{Szn09c} and \cite{Szn09d}, Sznitman proves results analogous to Theorem~\ref{t:dom} for random walk on a discrete cylinder for an analysis of disconnection times.

\medskip

We now comment on the proofs, beginning with Theorem~\ref{t:dom}. In order to convey the idea behind the proof at an intuitive level, we briefly describe a construction of the law of $\mathcal{I}^u \cap \ff A$, i.e. of the interlacement set at level $u$ inside a box $\ff A \subset \mathbb{Z}^d$ (for details, see Section~\ref{s:prel}): Consider first a Poisson random variable $J$ with parameter $u \, \text{cap}(\ff A)$, then run $J$ independent random walks starting at vertices distributed according to the normalized equilibrium measure $P_{e_{\ff A}}/\text{cap}(\ff A)$ (this distribution can be thought of as the hitting distribution of $\ff A$ by a random walk started at infinity, see \eqref{d:cap} for the definition). The trace left by these $J$ random walk trajectories in $\ff A$ has the same law as $\mathcal{I}^u \cap \ff A$.

\medskip

At a heuristic level, Theorem~\ref{t:dom} can now be understood as follows: the small ball $A \subseteq \mathbb{T}$ is only rarely visited by the random walk, so the total number of visits to it should approximately be Poisson distributed. By mixing properties of the random walk, the successive visits should be close to independent and start from a vertex in $A$ chosen roughly according to the normalized equilibrium measure on $\ff A$. Provided these approximations are valid, the trace of the successive visits to $A$ looks similar to $\mathcal{I}^u \cap \ff A$.

\medskip

Our proof of Theorem~\ref{t:dom} is inspired by \cite{Szn09c} and \cite{Szn09d}. In particular, it also consists of a poissonization and of a truncation step. We now describe these two steps.

\medskip

In the poissonization step, we need to identify suitable excursions of the random walk. These excursions should include all visits to $A$ made by the random walk and should be comparable with independent random walk paths entering $A$ (for the moment, we are not asking for the entrance points to have distributions similar to $P_{e_{\ff A}}/\text{cap}({\ff A})$). Unlike the discrete cylinder considered in \cite{Szn09c} and \cite{Szn09d}, the torus provides no natural geometric structure with respect to which appropriate excursions can be defined. Instead, each of our random walk excursions is defined to start by entering the ball $A$ and to end as soon as the random walk has spent a time interval of length $(N \log N)^2$ outside of a larger ball $B = B(0,N^{1-\epsilon/2}) \supset A$. We show that the distribution of the position of the random walk upon completion of such an excursion is close to the quasistationary distribution with respect to $B$, see Lemma~\ref{l:quasi} (due to periodicity issues, we work with the continuous-time random walk for this part of the argument). As a result, we can deduce in Lemma~\ref{l:decouple} that successive excursions are close to independent. In Proposition~\ref{p:c1}, we then use these observations to construct a coupling of the random walk trajectory with two Poisson random measures on the space of trajectories in the torus, such that the trace of the random walk paths in $A$ is bounded from above and from below by the traces of the Poisson random measures. With the estimate derived in Lemma~\ref{l:quni} on the hitting distribution of $A$ by the random walk started from the quasistationary distribution, we can modify this coupling in Proposition~\ref{p:c2}, such that the random trajectories appearing in the Poisson measures all start from the normalized equilibrium measure on ${\ff A} = \phi(A)$. 

\medskip



Finally, we come to the truncation step. The deficiency of the random measures described in the last paragraph is that, due to the finiteness of the torus, the appearing random excursions do not have the same
distributions as random walks in ${\mathbb Z}^d$. In the truncation step, we prove that it is possible to control the traces of these Poisson random measures in $A$ from above and from below by random interlacements with slightly changed intensities. This is achieved by truncation and sprinkling arguments from \cite{Szn09c} and \cite{Szn09d} with some modifications due to the different definition of our excursions.

\medskip

The proofs of the applications of Theorem~\ref{t:dom} roughly employ the following heuristics: we first reduce the proof of a global statement such as $|\mathcal{C}^u_{max}| \geq \epsilon N^d$ to several so-called `local estimates'. We use the term `local' to describe events which only depend on the configuration of visited sites inside a box of radius $N^{1-\epsilon}$ in $\mathbb{T}$. After this reduction, the desired results can be established using Theorem~\ref{t:dom}, together with known results on interlacements percolation. A more detailed description of the above strategy, together with the complete proofs of Theorems~\ref{t:torus}, \ref{t:d3} and \ref{t:super}, can be found in Section~\ref{s:not}. 

\medskip

The article is organized as follows: In Section~\ref{s:not}, we introduce some notation and use Theorem~\ref{t:dom} to prove Theorems~\ref{t:torus}, \ref{t:d3} and \ref{t:super}. The other Sections~\ref{s:prel}-\ref{s:dom-} prove Theorem~\ref{t:dom}. Section~\ref{s:prel} contains preliminary estimates on expected entrance times and the required properties of the quasistationary distribution. The poissonization of the random walk trace is performed in Section~\ref{s:dom} and the truncation and resulting coupling with random interlacements in Sections~\ref{s:dom+} and \ref{s:dom-}.

\medskip

Finally, we use the following convention concerning constants: Throughout the text, $c$ or $c'$ denote strictly positive constants depending only on $d$, with values changing from place to place. Dependence of constants on additional parameters appears in the notation. For example, $c_\alpha$ denotes a constant depending only on $d$ and $\alpha$.

\medskip

\paragraph{\textbf{Acknowledgments.}} The authors are grateful to Alain-Sol Sznitman for helpful discussions. A significant part of this work was accomplished when Augusto Teixeira was visiting the Weizmann Institute of Science and when David Windisch was visiting ETH Zurich. The authors would like to thank the Weizmann Institute and the FIM at ETH for financial support and hospitality during these visits. Augusto Teixeira's research has been supported by the grant ERC-2009-AdG 245728-RWPERCRI.

%% file: notation.tex
\section{Applications}
\label{s:not}

In this section we prove Theorems~\ref{t:torus}, \ref{t:d3} and \ref{t:super} which are the main applications of Theorem~\ref{t:dom} that we present in this paper. But first, let us introduce some notation.

\medskip

We consider the lattice ${\mathbb Z}^d$ and the discrete integer torus ${\mathbb T} = {\mathbb T}_N = ({\mathbb Z}/N{\mathbb Z})^d$, $d \geq 3$ ($N$ generally omitted), both equipped with edges between any two vertices at Euclidean distance $1$. For vertices $x, y$, we write $x \sim y$ to state that $x$ and $y$ are neighbors. For any vertex $x$ and $r \geq 0$, $B(x,r)$ denotes the closed ball centered at $x$ with radius $r$ with respect to the $\ell_\infty$-distance. The canonical projection from ${\mathbb Z}^d$ to $\mathbb T$ mapping $(x_1, \ldots, x_d)$ to $(x_1 \textup{ mod } N, \ldots, x_d \textup{ mod } N)$ is denoted $\Pi$. Given $x \in {\mathbb T}$, we introduce the bijection $\phi_x$ from $B(x,N/4) \subset {\mathbb T}$ to $B(0,N/4) \subset {\mathbb Z}^d$ satisfying $\Pi(\phi_x(x + x'))=x'$ for any $x' \in B(0,N/4) \subset {\mathbb T}$, and for simplicity of notation write $\phi$ for $\phi_0$. For any subsets $A, B, C, \ldots$ of $B(0,N/4) \subset {\mathbb T}$, we generally write ${\mathsf A} = \phi(A), {\mathsf B} = \phi(B)$ and ${\mathsf C} = \phi(C)$. Random sets of vertices are generally denoted $\mathcal A$, $\mathcal B$ and $\mathcal C$. For any set $V$ of vertices, the internal boundary $\partial_i V$ is defined as the set of vertices in $V$ with at least one neighbor in $V^c$, while the external boundary is denoted $\partial_e V = \partial_i (V^c)$. If $V$ is finite, we denote its cardinality by $|V|$ and its diameter with respect to the $\ell_\infty$-distance by $\diam (V)$. For real numbers $a$ and $b$, we write $a \wedge b$ for the minimum and $a \vee b$ for the maximum in $\{a,b\}$. Equations involving the symbol $\pm$ stand for two separate equations, one with $+$, one with $-$. For example, $\aleph_\pm = \beth_\pm$ is short-hand notation for $\aleph_+ = \beth_+$, $\aleph_- = \beth_-$.

\medskip

Finally, we write $P_x$ for the law on ${\mathbb T}^{\mathbb N}$ of the simple random walk on $\mathbb T$ started at $x \in {\mathbb T}$, and denote the canonical coordinate process by $(X_n)_{n \geq 0}$, where by simple random walk on $\mathbb T$ we mean the projection of the canonical simple random walk on ${\mathbb Z}^d$ under $\Pi$. We use $P$ to denote the law with uniformly chosen starting point, i.e. $P = \sum_{x \in {\mathbb T}} N^{-d} P_x$.

\medskip

The random interlacements $({\mathcal I}^u)_{u \geq 0}$ at levels $u \geq 0$ are all defined on a suitable probability space $(\Omega, {\mathcal F}, {\mathbb P})$, see \cite{Szn09} for details. For $x \in \mathbb{Z}^d$, we denote by $\mathcal{C}^u_x$, the connected component of $\mathcal{V}^u$ containing $x$ (cf.~\eqref{e:Vu}). We also use the same notation ($\mathcal{C}^u_x$) to denote the connected component of $\mathbb{T}^d \setminus X_{[0,uN^d]}$ containing $x$, but the two cases can be distinguished by the context. The event that there is a nearest-neighbor path from vertex $x \in {\mathbb Z}^d$ to vertex $y \in {\mathbb Z}^d$ using only vertices in ${\mathcal V}_u$ is denoted $\{x \stackrel{{\mathcal V}_u}{\longleftrightarrow} y\}$.

\medskip

Using only Theorem~\ref{t:dom} and known results on random interlacements, we now prove Theorems~\ref{t:torus}, \ref{t:d3} and \ref{t:super}, establishing the existence and some properties of the distinct phases for the sizes of components left by the simple random walk on $\mathbb T$. 

\medskip

The value $u_{\star \star}$ in the statement of Theorem~\ref{t:torus} is defined as in \cite{SS10} as follows:
\begin{equation}
\label{e:udouble}
\begin{array}{c}
u_{\star \star} = \inf\{u \geq 0; \alpha(u) > 0\}, \text{ where}\\
\alpha(u) = \sup\Big\{\alpha \geq 0; \lim_{L \to \infty} L^\alpha \mathbb{P}[B(0,L) \overset{\mathcal{V}^u}{\longleftrightarrow} \partial_e B(0,2L)] = 0\Big\}, \text{ for $u > 0$}.
\end{array}
\end{equation}
where by convention the supremum of an empty set is zero. It is shown in \cite{SS10}, Theorem~0.1 that there is a constant $\kappa>0$ depending only on $d$ and $u$, such that
\begin{equation}
\label{e:int2}
\text{for any $d \geq 3$ and } u > u_{\star \star}, \, Q^u [ 0 \stackrel{{\mathcal V}_u}{\longleftrightarrow} x] \leq c_u \exp (- c_u' |x|^\kappa),
\end{equation}
for $u_{**}$ defined in \eqref{e:udouble}.

\begin{remark}
It is currently unknown whether $u_\star$ differs from $u_{\star \star}$. If it turns out that these two values are in fact equal, then \eqref{e:torus2} will make \eqref{e:torus1} obsolete.
\end{remark}

\medskip

In the proof of Theorem~\ref{t:torus}, we first use the Markov inequality to reduce the desired tail estimates on $|{\mathcal C}_{max}|$ to tail estimates on the size of the vacant component containing $0$, intersected with the ball $B(0,N^{1-\epsilon})$. Theorem~\ref{t:dom} then allows us to deduce such tail estimates from bounds on the finite clusters in the random interlacements model. We obtain a stronger bound for $u>u_{\star \star}$, thanks to the strong connectivity decay guaranteed by this assumption.

\medskip

\begin{proof}[Proof of Theorem~\ref{t:torus}.] We start with the proof of \eqref{e:torus1}. For this we take any $\alpha, \epsilon \in (0,1)$ such that $u(1-\epsilon) > u_\star$. Write $A_x$ for $B(x,N^{1-\epsilon})$ and recall that ${\mathcal C}_x^u$ stands for the connected component of $\mathbb{T} \setminus X_{[0,uN^d]}$ containing $x \in \mathbb{T}$.
For $N \geq c_{\eta, \epsilon}$, we have $\eta N^d > |A_x|$, therefore, by the Chebychev inequality,
\begin{equation}
\label{e:Cmaxk}
\begin{split}
P[|{\mathcal C}_{max}^u| \geq \eta N^d] &= P \biggl[ \sum_{x \in \mathbb{T}} \mathbf{1}_{\{|{\mathcal C}_x| \geq \eta N^d\}} \geq \eta N^d \biggr] \leq \frac{1}{\eta N^d} \sum_{x \in \mathbb{T}} P[|{\mathcal C}_x| \geq \eta N^d] \\
&\leq \frac{1}{\eta N^d} \sum_{x \in \mathbb{T}} P \big[x \stackrel{\mathbb{T} \setminus X_{[0,uN^d]}}{\longleftrightarrow} \partial_i A_x \big],
\end{split}
\end{equation}
because a component of size $\eta N^d > |A_x|$ cannot be contained in $A_x$. Considering the coupling $Q$ provided by Theorem~\ref{t:dom}, the probability appearing in the last line of the inequality above equals
\begin{equation}
Q \Big[0 \stackrel{\mathsf X(u,\mathsf A)^c}{\longleftrightarrow} \partial_i B(0,N^{1-\epsilon}) \Big] \leq Q \big[0 \stackrel{\mathcal V^{u(1-\epsilon)}}{\longleftrightarrow} \partial_i B(0,N^{1-\epsilon})\big] + c_{u, \alpha, \epsilon} N^{-\alpha}.
\end{equation}
By the definition of $u_\star$, the continuity of probability measures and the fact that $u(1-\epsilon) > u_\star$, the above probability converges to zero as $N$ goes to infinity. Thus, we conclude \eqref{e:torus1}.

For the proof of \eqref{e:torus2}, given $u>u_{\star \star}$, take $\epsilon>0$ such that $(1-\epsilon)u > u_{\star \star}$ and choose $\lambda = 2d/\kappa$, c.f. \eqref{e:int2}. Note that $\lambda$ depends solely on $d$ and $u$. We use \eqref{e:int2} and the fact that every set $D \subset \mathbb{Z}^d$ (or $\mathbb{T}$) satisfies $\diam(D) \geq c |D|^{1/d}$ (for some $c > 0$), to obtain 
\begin{equation}
\label{e:smaller}
Q^u [ |\mathcal{C}^u_0| > \log^\lambda N] \leq Q^u [ \diam(\mathcal{C}^u_0) > c_u \log^{2/\kappa} N] \leq c_u' \exp (- c_u \log^2 N).
\end{equation}

We note that for $N \geq c_{\lambda, \epsilon}$, we have $\log^\lambda N < \frac{1}{2}N^{1-\epsilon}$ and hence $\{|\mathcal{C}^u_x| > \log^\lambda N\} \subseteq \{|\mathcal{C}^u_x \cap A_x| > \log^\lambda N\}$. Thus, by Theorem~\ref{t:dom} (with $\alpha = \rho + d$), we have
\begin{equation*}
P[|\mathcal{C}^u_{max}| > \log^\lambda N] \leq \sum_{x \in \mathbb{T}} P[|\mathcal{C}^u_x \cap A_x| > \log^\lambda N] \leq N^d (Q^u[|\mathcal{C}^u_0| > \log^\lambda N] + c_{u,\rho}N^{-\alpha}).
\end{equation*}
And we conclude \eqref{e:torus2} from \eqref{e:smaller}.
\end{proof}

Having proved the absence of a macroscopic component in the large $u$ regime, we now proceed with the small $u$ case.

\medskip

We now briefly describe the idea of the proof of Theorem~\ref{t:d3}. We first slice the torus $\mathbb{T}$ into $N^{d-2}$ parallel planes denoted by $\{F_x\}_{x \in ({\mathbb Z}/N{\mathbb Z})^{d-2}}$. Although our argument works for all $d \geq 3$, it is instructive to keep in mind the picture in the special case $d = 3$. Using the link with random interlacements from Theorem~\ref{t:dom}, together with known results on random interlacements, we show that with high probability, any such plane $F_x$ contains no occupied dual path longer than $\sqrt{N}/2$. By a geometric argument, we show that under these conditions, all vertices in in vacant components of $F_x$ with diameter at least $\sqrt{N}/2$ belong to the same component of $\mathbb T$ (we call such vertices `seeds'). Finally, we use Proposition~\ref{p:average} below to show that the number of seeds in $\mathbb{T}$ is at least $\epsilon N^d$.

\medskip

Let us now prove Proposition~\ref{p:average}. Roughly speaking, it states that if a given increasing event has positive probability under the random interlacements law and solely depends on what happens in a fixed box, then with high $P$-probability this event will be observed simultaneously in various boxes in the torus. We first need to introduce Definition~\ref{d:pull}, where for $\delta \in (0,1)$, we write $B_{x,\delta} = B(x,N^\delta/2) \subset {\mathbb T}$ and $\ff B_{x,\delta} = B(x,N^\delta/2) \subset {\mathbb Z}^d$.

\begin{definition}
\label{d:pull}
For a given function $f:2^{\mathbb{Z}^d} \to [0,1]$, measurable with respect to the Borel-$\sigma$-algebra on $[0,1]$ and the canonical $\sigma$-algebra on $2^{{\mathbb Z}^d}$ generated by the coordinate projections, and some $x \in \mathbb{T}$ (respectively, $x \in \mathbb{Z}^d$), we define the \textit{local pullback} $f^x_N:2^{\mathbb{T}} \to [0,1]$ (respectively, $f^x_N:2^{\mathbb{Z}^d} \to [0,1]$) by
\begin{equation}
f_N^x({U}) = \left\{ \begin{array}{ll} 
f\Big(\phi_x\big({U}\cap B_{x,\delta} \big)\cup \ff B_{0,\delta}^c\Big), & \text{for } x \in {\mathbb T}, U \subseteq {\mathbb T}, \\
f\Big( \big( \big({U}\cap {\mathsf B}_{x,\delta} \big) -x \big)\cup \ff B_{0,\delta}^c\Big), & \text{for } x \in {\mathbb Z}^d, U \subseteq {\mathbb Z}^d, 
\end{array} \right.
\end{equation}
where $\phi_x$ is the isomorphism between $B(x,N/4) \subset {\mathbb T}$ and $B(0,N/4) \subset {\mathbb Z}^d$ defined in the second paragraph of this section.
\end{definition}

\begin{proposition}
\label{p:average}
\textup{($d \geq 3$)}
Consider $\beta > 0$, $\delta \in (0,1)$ and let $f$ be a monotone non-decreasing function $f:2^{\mathbb{Z}^d} \to [0,1]$, such that for some $k > 0$, 
\[\alpha_1 := \mathbb{E}\big[f(\mathcal{V}^{u(1+\beta)})\big] \;\; \leq \;\; \alpha_2 := \mathbb{E}\big[f(\mathcal{V}^{u(1-\beta)} \cup B(0,k)^c)\big], \]
where $\mathbb E$ denotes $\mathbb P$-expectation (cf.~the beginning of this section).
Then for any $\epsilon > 0$,
\begin{equation}
\lim\limits_{N \to \infty} P \Big[ \alpha_1 - \epsilon \leq \bar f_N \leq \alpha_2 + \epsilon \Big] = 1,
\end{equation}
where $\bar f_N$ is the average of the local pullbacks: $\bar f_N = \frac{1}{N^d}\sum_{x \in \mathbb{T}} f^x_N\big(\mathbb{T} \setminus X_{[0,uN^d]}\big)$, see Definition~\ref{d:pull}.
\end{proposition}

\begin{proof}
In this proof we omit the indices $\delta$ and $N$ from $B_{x,\delta}$, $f^x_N$ and $\bar f_N$. We define the local average $N_x^u$ by
\[N_x^u = \frac{1}{|\ff B_x|} \sum_{y \in {\mathsf B}_x} f^y\big(\mathcal{V}^u\big) \; \text{ or } \; N_x^u = \frac{1}{|B_x|} \sum_{y \in B_x} f^y\big(\mathbb{T} \setminus X_{[0,uN^d]}\big),\]
depending on whether $x$ belongs to $\mathbb{Z}^d$ or $\mathbb{T}$. Note that $N_x^u$ is monotone non-decreasing and it only depends on the configuration inside $B(x,N^\delta)$.

Monotonicity of $f$ implies that, if $x \in \mathbb{Z}^d$ and $N^\delta/2 \geq k$,
\[f\big(\mathcal{V}^{u(1+\beta)}-x\big) \leq  f^x(\mathcal{V}^{u(1+\beta)}) = f^x(\mathcal{V}^{u(1-\beta)}) \leq f\big((\mathcal{V}^{u(1-\beta)} \cup B(x,k)^c)-x\big).\]
Thus, we conclude that $\lim_{N \to \infty}\mathbb{P}\big[\alpha_1 - \epsilon/2 \leq N_0^{u(1+\beta)} \leq N_0^{u(1-\beta)} \leq \alpha_2 + \epsilon/2 \big] = 1$, using the fact that the set $\mathcal{V}^v$ is ergodic under translation maps, see \cite{Szn09}, Theorem~2.1. This implies, by Theorem~\ref{t:dom} and the monotonicity of $N^u_x$, that for any sequence $w_N \in \mathbb{T}_N$ (we omit the index $N$ in the notation below),
\begin{equation}
\lim\limits_{N\to \infty} P\Big[\alpha_1 - \epsilon/2 \leq {N_w^{u}} \leq \alpha_2 +\epsilon/2\Big] = 1.
\end{equation}

We now let $R = \{w \in \mathbb{T};  N_w^{u} \not \in (\alpha_1 - \epsilon/2, \alpha_2 + \epsilon/2)\}$ and note that $E[|R|]/N^d \to 0$ as $N \to \infty$. Which implies that $|R|/N^d$ converges in probability to zero as $N$ tends to infinity.

It is clear that $\bar f$ can also be written as $\bar f = \frac{1}{N^d} \sum_{w\in \mathbb{T}} {N^u_w}$, thus,
\begin{equation}
\label{e:Ndens}
\begin{split}
P\Big[ & \bar f \not \in (\alpha_1 - \epsilon, \alpha_2 + \epsilon) \Big]
 = P\Big[\sum_{w\in R} \tfrac{N^u_w}{N^d} + \sum_{w\in \mathbb{T} \setminus R} \tfrac{N^u_w}{N^d}   \not \in (\alpha_1 - \epsilon, \alpha_2 + \epsilon)\Big] \\
& \overset{0 \leq N_x^u \leq 1}{\leq} P\Big[ \sum_{w\in \mathbb{T} \setminus R} \tfrac{N^u_w}{N^d} \leq \alpha_1 - \epsilon \Big] + P\Big[\frac{|R|}{N^d} + \sum_{w\in \mathbb{T} \setminus R} \tfrac{N^u_w}{N^d} \geq \alpha_2+\epsilon \Big] \\
& {\leq} P\Big[ (\alpha_1 - \epsilon/2) \tfrac{|\mathbb{T} \setminus R|}{N^d} \leq \alpha_1 - \epsilon \Big] + P\Big[\frac{|R|}{N^d} \geq \epsilon/2\Big],
\end{split}
\end{equation}
which converges to zero as $N$ goes to infinity since $|R|/N^d$ converges in probability to zero. This proves Proposition~\ref{p:average}.
\end{proof}

In the proof of Theorem~\ref{t:d3}, we will show the existence of vacant crossings of two-dimensional planes in the torus, and then use a geometric argument to deduce the existence of a macroscopic component. To this end, we introduce the following notions: we define a $\star$-path to be a sequence of distinct points $x_1, \dots, x_k$ (in $\mathbb{Z}^d$ or in $\mathbb{T}$) such that $x_i$ and $x_{i+1}$ are at $\ell_\infty$-distance at most one, for all $i = 1, \dots, k-1$. It is known that there exists a $\tilde u > 0$, such that for all $u \leq \tilde u$,
\begin{equation}
\label{e:starpath2}
\lim\limits_{N\to \infty}N^{2d} \cdot \mathbb{P}[ \text{there is a $\star$-path in $\mathbb{Z}^2 \cap \mathcal{I}^{u}$ from $0$ to $\partial_i B(0,N^\epsilon/4)$}] = 0,
\end{equation}
see (3.28) of \cite{SS09}. Here $\mathbb{Z}^2 \subset \mathbb{Z}^d$ denotes the set of vertices with only the first two coordinates not equal to zero. Moreover, we can choose $\tilde u$ such that,
\begin{equation}
\label{e:percplan}
\mathbb{P}[0 \text{ belongs to an infinite cluster of $\mathcal{V}^u \cap \mathbb{Z}^2$}] > 0, \text{ for } 0 \leq u \leq {\tilde u}, 
\end{equation}
see Theorem~3.4 of \cite{SS09}.

\begin{proof}[Proof of Theorem~\ref{t:d3}.]
Consider any $0<u<{\tilde u}/2$. Given a point $x = (x_1,\dots, x_d) \in \mathbb{T}$, the set
\[F_x = \big\{y = (y_1,\dots,y_d) \in \mathbb{T}; y_i = x_i \text{ except for $i \in \{1,2\}$}\big\}\]
is called the horizontal plane though $x$.

We now fix $\epsilon = 1/2$. We say that a point $x \in \mathbb{T}$ is a \textit{seed} if $x$ is connected to $\partial_i B(x,N^\frac{1}{2}/2)$ though $F_x \setminus X_{[0,uN^d]}$, where $F_x$ is the horizontal plane passing through $x$. We say that a path (respectively a $\star$-path) in $F_x$ is \textit{projected}, if it is given by the image of a nearest neighbor (respectively $\star$-nearest neighbor) path in $\{0,\dots,N-1\}^2 \subset \mathbb{Z}^2$ under the map $(y_1,y_2) \mapsto (y_1,y_2,x_3,\dots,x_d)$. For instance, note that a jump from $(0,\dots,0)$ to $(N-1,0,\dots,0)$ is not allowed for a projected path. 
To establish the result, we need the following claim:
\begin{equation}
\label{e:joinseed}
\begin{array}{c}
\text{if for every horizontal plane $F_x \subset \mathbb{T}$, the longest $\star$-path in $F_x \cap X_{[0,uN^d]}$} \\
\text{has diameter smaller or equal to $N^\frac{1}{2}/2$, then $|\mathcal{C}^u_{max}| \geq |\{\text{seeds in $\mathbb{T}$}\}|$.}
\end{array}
\end{equation}

We first introduce the following definition:
\begin{equation}
\label{e:cross}
\begin{array}{c}
\text{we say that a connected set $\mathcal{C} \subseteq F_x$ has a \textit{crossing} in} \\
\text{the plane $F_x$, if one can find two projected paths in $\mathcal{C}$,} \\
\text{crossing the square $F_x$ along the vertical and horizontal directions.}
\end{array}
\end{equation}

Consider a horizontal plane $F_x$ as in \eqref{e:joinseed}. Since there is no $\star$-path in $F_x \cap X_{[0,uN^d]}$ with diameter strictly greater than $N^\frac{1}{2}/2$, there is no projected $\star$-path in $F_x \cap X_{[0,uN^d]}$ connecting two opposite sides of $F_x$. By a duality argument (see \cite{K81} Proposition~2.2 p.30 and Example~(i) p.18) this implies that $F_x \setminus X_{[0,uN^d]}$ has a component $\mathcal{C}_{F_x}$ with a crossing in the sense of \eqref{e:cross}. 

We now show that every seed $x$ of $F_x$ is contained in $\mathcal{C}_{F_x}$. For this, suppose that $\mathcal{C}_{F_x} \neq \mathcal{C}_x$ and let $\bar{\mathcal{C}}_1$ and $\bar{\mathcal{C}}_2$ be the preimages under the projection $\Pi:\mathbb{Z}^d \to \mathbb{T}$ of the components $\mathcal{C}_{F_x}$ and $\mathcal{C}_x$ respectively. By considering separately the case in which at least one of these sets has unbounded components, or both have only bounded components, we can use Proposition~2.1, p.~387, of \cite{K81} to find a $\star$-path of diameter at least $N^\frac{1}{2}/2$ in $F_x \cap X_{[0,uN^d]}$, which contradicts the hypothesis in \eqref{e:joinseed}. Hence, every seed in $F_x$ is contained in ${\mathcal C}_{F_x}$.

To conclude the proof of \eqref{e:joinseed}, we prove that for any pair of horizontal planes $F_x$, $F_y \subset \mathbb{T}$, the components $\mathcal{C}_{F_x}$ and $\mathcal{C}_{F_y}$ are connected by a path in $\mathbb{T} \setminus X_{[0,uN^d]}$. It is enough to show this in the case where $F_x$ and $F_y$ are adjacent to each other. Indeed, once we obtain this result for adjacent horizontal planes we can extend it to every pair $F_x$, $F_y$ by considering a sequence $F_x = F_{x_0}, F_{x_1},\dots,F_{x_{k-1}},F_{x_k}= F_{y}$ of adjacent horizontal planes.

So, consider two adjacent horizontal planes $F_x$ and $F_y \subset \mathbb{T}$, meaning that every point in $F_x$ has exactly one neighbor in $F_y$. Recall that $\mathcal{C}_{F_x}$ contains a projected path $\tau = x_0, \dots, x_k$ joining the top and the bottom sides of $F_x$. The respective neighbors $y_1,\dots,y_k$ of $x_1,\dots,x_k$ in $F_y$ also constitute a projected path $\tau'$ joining the top and bottom sides of $F_y$. By the fact that $F_y$ is crossed (from side to side) by projected paths in $\mathcal{C}_{F_y}$, we obtain that $\tau'$ meets $\mathcal{C}_{F_y}$ and therefore $\mathcal{C}_{F_x}$ and $\mathcal{C}_{F_y}$ are connected, see Figure~\ref{f:join}. This finishes the proof that every seed belongs to the same connected component of $\mathbb{T} \setminus X_{[0,uN^d]}$, hence of \eqref{e:joinseed}.

\begin{figure}
\begin{center}
\includegraphics[angle=0, width=0.5\textwidth]{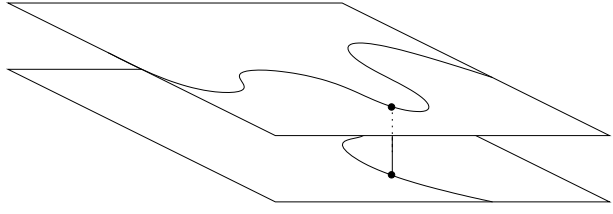}\\\label{f:join}
\caption{A connection between $\mathcal{C}_F$ and $\mathcal{C}_{F'}$.}
\end{center}
\end{figure}

We now have two remaining steps to establish Theorem~\ref{t:d3}. We need to show that the hypothesis in \eqref{e:joinseed} holds with high probability and that the number of seeds is, with overwhelming probability, larger or equal to $\epsilon N^d$ for some $\epsilon > 0$. We start proving the second part.

Recall that $u$ was chosen in a way that $u(1+\frac{1}{2}) < \tilde u$. Let $f:2^{\mathbb{Z}^d} \to [0,1]$ be given by $\ff{U} \mapsto 1\{0 \text{ belongs to an infinite cluster of $\ff{U} \cap \mathbb{Z}^2$}\}$. Using \eqref{e:percplan} we conclude that $\mathbb{E}[f(\mathcal{V}^{u(1+\frac{1}{2})})] := \gamma > 0$. The local pullback function $f^x_N$ of $f$ (see Definition~\ref{d:pull}) happens to be the indicator function that $x$ is a seed. Using Proposition~\ref{p:average} we obtain that
\begin{equation}
\label{e:seeds}
\lim\limits_{N\to\infty}P[ |\{\text{seeds in $\mathbb{T}$}\}| \geq (\gamma/2)N^d ] = 1.
\end{equation}

In view of the above and \eqref{e:joinseed}, to finish the proof of Theorem~\ref{t:d3} it is now enough to show that
\begin{equation}
\label{e:starpath}
\lim\limits_{N \to \infty}P \Big[
\begin{array}{c}
\text{for some horizontal plane $F \subset \mathbb{T}$, there is a $\star$-path} \\
\text{in $F \cap X_{[0,uN^d]}$ with diameter strictly larger than $N^\frac{1}{2}/2$, }
\end{array}\Big] = 0.
\end{equation}
For $u < \tilde u$, this probability is smaller or equal to
\begin{equation*}
\begin{split}
& N^d P[ \text{there is a $\star$-path in $F_x \cap X_{[0,uN^d]}$ from $x$ to $\partial_i B(x,N^\frac{1}{2}/4)$}] \\
& \overset{\text{Theorem~\tiny{\ref{t:dom}}}}{\leq} N^d \Big( \mathbb{P}[ \text{there is a $\star$-path in $\mathbb{Z}^2 \cap \mathcal{I}^{u(1+\frac{1}{2})}$ from $0$ to $\partial_i B(0,N^\frac{1}{2}/4)$}] + cN^{-2d} \Big).
\end{split}
\end{equation*}
The last term converges to zero as $N$ tends to infinity, due to \eqref{e:starpath2}. This finishes the proof of \eqref{e:starpath}, which together with \eqref{e:joinseed} and \eqref{e:seeds} establishes Theorem~\ref{t:d3}.
\end{proof}

We are now going to prove Theorem~\ref{t:super}, which is a stronger characterization of supercriticality. The estimates provided by this theorem hold for so-called strongly supercritical values of $u$ (cf.~Theorem~\ref{t:super}), which we introduce now.

\begin{definition}
\label{d:strong}
We say that $u \geq 0$ is strongly supercritical if there is a $\mu > 0$ such that, for large enough $N$ depending on $\mu$ and $u$, we have:
\begin{equation}
\label{e:touch}
\mathbb{P} \big[\text{there is a path in $\mathcal{V}^{u(1+\mu)}$ from $B(0,N)$ to infinity}\big] \geq 1 - e^{-N^\mu}, \text{ and }
\end{equation}
\begin{equation}
\label{e:conn}
\mathbb{P}\Big[
\begin{array}{c}
\text{any two connected subsets of $\mathcal{V}^{u(1-\mu)} \cap B(0,N)$ with} \\
\text{diameter $\geq N/8$ are connected through $\mathcal{V}^{u(1+\mu)} \cap B(0,2N)$}
\end{array}
\Big] \geq 1 - e^{-N^\mu}.
\end{equation}
\end{definition}

\begin{remark} \label{r:strong}
1) Note that the above mentioned connected sets need not be whole connected components of $\mathcal{V}^{u(1-\mu)}$. It is also important to note that \begin{equation}
\label{e:open}
\text{the set $\{u > 0; \text{ $u$ is strongly supercritical}\}$ is open.}
\end{equation}
To see this it is enough to note that under $\mathbb{P}$, $\mathcal{V}^u \subseteq \mathcal{V}^{u'}$ whenever $u \geq u'$.


%

2) It is important to note that for $d \geq 5$, one can prove the existence of some $\bar u(d) > 0$ such that every $u \leq \bar u(d)$ is strongly supercritical, see Theorems~3.2 and 3.3 of \cite{T09b}. We do not know if this holds for $d = 3,4$ (this is the main motivation for Theorem~\ref{t:d3}). It is an important question whether every $u < u_\star$ is strongly supercritical.

3) It is clear that if $u$ is strongly supercritical, and $\mu > 0$ is chosen as in Definition~\ref{d:strong}, then
\begin{equation}
\label{e:diamest}
\mathbb{P}[\text{$0$ is connected to $\partial_i B(0,2N)$ through $\mathcal{V}^{u(1+\mu)}$, but not to infinity}] \leq 2e^{-N^\mu},
\end{equation}
for $N$ large enough depending on $u$ and $\mu$.
\end{remark}

\begin{proof} [Proof of Theorem~\ref{t:super}.]
Theorem~\ref{t:super} follows from Propositions~\ref{p:sec} and \ref{p:density} below.
\end{proof}

We first need the following deterministic lemma, which gives a local criterion implying that a given set has a unique giant component.

\begin{lemma}
\label{l:sec}
\textup{($d \geq 3$)}
Consider $\ell \leq N/10$ and $A \subseteq \mathbb{T}$, such that for every $x \in \mathbb{T}$,
\begin{enumerate}
\item the set $A \cap B(x,2\ell)$ has a connected component with diameter at least $\ell$,
\item every pair of components in $A \cap B(x,6 \ell)$ with diameter at least $\ell$ belong to the same component of $A$.
\end{enumerate}
Then there exists a unique component of $A$ with diameter bigger or equal to $\ell$.
\end{lemma}

We stress the similarity between the hypotheses above and Definition~\ref{d:strong}. Note that Lemma~\ref{l:sec} reduces the task of bounding $P[|\mathcal{C}^u_{sec}| > \log^\lambda N]$ to local estimates, which we then perform using Theorem~\ref{t:dom} and the hypothesis that $u$ is strongly supercritical.

\begin{proof}[Proof of Lemma~\ref{l:sec}]
Recall that $\Pi$ stands for the canonical projection from $\mathbb{Z}^d$ to $\mathbb{T}$ and write $\bar{\ff A}$ for $\Pi^{-1} (A)$. Consider a paving $\{\ff B_i\}_{i\in I}$ of $\mathbb{Z}^d$ with boxes of radius $2\ell$ (the reason why we work with $\mathbb{Z}^d$ instead of $\mathbb{T}$ is to simplify this paving procedure). For every such $\ff B_i$, we choose, using some arbitrary order, a component $\ff C_i$ of $\bar{\ff A} \cap \ff B_i$ with diameter at least $\ell$: the existence of such component follows from Hypothesis~\textit{1} of Lemma~\ref{l:sec} and the fact that $\ff B_i$ and $\Pi(\ff B_i)$ are isometric.

We claim that all components $\ff C_i$ belong to the same connected component of $\bar{\ff A}$. To see this, note that for every pair $\ff C_i$ and $\ff C_{i'}$, one can find a path of adjacent boxes $\ff B_{i_1}, \dots, \ff B_{i_k}$ in $\{\ff B_i\}_{i \in I}$ such that $\ff C_i \subset \ff B_{i_1}$ and $\ff C_{i'} \subset \ff B_{i_k}$. Then, we use Hypothesis~\textit{2} of Lemma~\ref{l:sec} for each pair of consecutive boxes in the mentioned path to conclude that $\ff C_{i_j}$ and $\ff C_{i_{j+1}}$ are connected through $\bar{\ff A}$, for $i = 1, \dots, k-1$. This shows that $\ff C_i$ and $\ff C_{i'}$ can be connected through $\bar{\ff A}$. Since $i$ and $i'$ were arbitrarily chosen, we conclude that all $\{\ff C_i\}_{i \in I}$ belong to the same connected component of $\bar{\ff A}$.

Consider now some fixed $i \in I$ and denote by $C \subset \mathbb{T}$ the connected component of $A$ containing $\Pi(\ff C_i)$. The fact that all $\ff C_i$'s belong to the same connected component of $\bar{\ff A}$ implies that, for every $x \in \mathbb{T}$, $C \cap B(x,6 \ell)$ has a connected component of diameter at least $\ell$. Let $C'$ be any connected component of $A$ of diameter larger or equal to $\ell$, possibly different from $C$. Then there is a point $x$ in $\mathbb{T}$ for which $C' \cap B(x,6 \ell)$ has a component of diameter at least $\ell$. Since $C \cap B(x,6 \ell)$ also has a component of diameter at least $\ell$, by Hypothesis~\textit{2} of Lemma~\ref{l:sec} we have that $C$ and $C'$ are the same. This proves that $C$ is the unique component of $A$ with diameter greater or equal to $\ell$.
\end{proof}

\begin{proposition}
\label{p:sec}
If $u$ is strongly supercritical, then for some $\lambda = \lambda(u) > 0$ and every $\rho > 0$, there exists a constant $c=c(d,u,\rho)>0$, such that
\begin{equation}
\label{e:sec2}
P [|\mathcal{C}^u_{sec}| > \log^\lambda N] \leq c N^{-\rho},
\end{equation}
where $|\mathcal{C}^u_{sec}|$ denotes the volume of the second largest component of $\mathbb{T} \setminus X_{[0,uN^d]}$ (cf.~above Theorem~\ref{t:super}).
\end{proposition}

\begin{proof}
We are going to make use of Lemma~\ref{l:sec} and Theorem~\ref{t:dom}. Fix a strongly supercritical intensity $u > 0$, take $\mu = \mu(u) > 0$ as in Definition~\ref{d:strong} and choose $\lambda = 4d/\mu$. If $\diam (\bar{\mathcal{C}}_{sec}^u)$ denotes the second largest diameter among the components of $\mathbb{T} \setminus X_{[0,uN^d]}$, the comment above \eqref{e:smaller} implies that $\diam(\bar{\mathcal{C}}^u_{sec}) \geq \diam(\mathcal{C}_{max}^u) \wedge \diam(\mathcal{C}^u_{sec}) \geq c|\mathcal{C}^u_{sec}|^{1/d}$. Therefore, by the comment above \eqref{e:smaller} and Lemma~\ref{l:sec}, we have 
\begin{equation}
\label{e:boundsec}
\begin{split}
&P[ |\mathcal{C}^u_{sec}| > \log^\lambda N]  \overset{N \geq\, c_{u,\mu}}\leq P[\diam(\bar{\mathcal{C}}^u_{sec}) > \log^{2/\mu} N]  \\
&\leq  P\Big[
\begin{array}{c}
\text{there is some $x \in \mathbb{T}$, such that all components in} \\
\text{$B(x,2\log^{2/\mu} N)\setminus X_{[0,uN^d]}$ have diameter smaller than $\log^{2/\mu} N$}
\end{array}\Big] + \\
&  P\Big[
\begin{array}{c}
\text{there is an $x \in \mathbb{T}$ and two components of $B(x,6\log^{2/\mu} N \setminus X_{[0,uN^d]})$} \\
\text{with diameters at least $\log^{2/\mu} N$, which are not connected in $\mathbb{T} \setminus X_{[0,uN^d]}$}
\end{array}\Big].
\end{split}
\end{equation}
According to Theorem~\ref{t:dom} (with $\alpha = \rho + d$, $\epsilon = \mu$), if $N \geq c_u$, the first term in the sum above is bounded by
\begin{equation}
N^d\Big(\mathbb{P}\Big[B(0,\log^{2/\mu} N) \overset{\mathcal{V}^{u(1+\mu)}}{\nleftrightarrow} \infty\Big] + c_{u, \rho} N^{-\alpha}\Big),
\end{equation}
which, in view of \eqref{e:touch}, is smaller or equal to $c_{u,\rho}N^{-\rho}$. We again use Theorem~\ref{t:dom} (with $\alpha = \rho + d$, $\epsilon = \mu$) to bound (for $N \geq c_u$) the second term on the right hand side of \eqref{e:boundsec} by
\begin{equation}
N^d \Bigg( \mathbb{P}\Bigg[
\begin{array}{c}
\text{there are two connected subsets of} \\
\text{$\mathcal{V}^{u(1-\mu)} \cap B(0,6 \log^{2/\mu}N)$ with diameters at least $\log^{2/\mu}N$} \\
\text{which are not connected in $\mathcal{V}^{u(1+\mu)} \cap B(0,12 \log^{2/\mu}N)$}
\end{array}\Bigg] \negthickspace + c_{\mu,\rho}N^{-\alpha}  \negthickspace \Bigg).
\end{equation}
By \eqref{e:conn}, this term is also bounded by $c_{u,\rho} N^{-\rho}$. This proves Proposition~\ref{p:sec}.
\end{proof}

\begin{proposition}
\label{p:density}
If $u$ is strongly supercritical, then for every $\epsilon > 0$,
\begin{equation}
\label{e:density2} 
\lim\limits_{N\to\infty}P\Big[ \Big|\frac{|\mathcal{C}^u_{max}|}{N^d} - \eta(u) \Big| > \epsilon \Big] = 0.
\end{equation}
\end{proposition}

\begin{proof}
Let $f:2^{\mathbb{Z}^d} \to [0,1]$ be given by $\ff{U} \mapsto 1\{0 \text{ belongs to an infinite cluster of $\ff{U}$} \}$ and choose any $\delta \in (0,1)$. By the continuity of the function $\eta$ in $[0,u_\star)$ (recall the definition in \eqref{e:eta(u)} and see \cite{T08}, Corollary~1.2), for small enough $\beta > 0$ and large enough $k$, we have
\begin{equation}
\label{e:Ebound} 
\eta(u) - \epsilon/2 \leq \mathbb{E}[f(\mathcal{V}^{u(1+\beta)})] \leq \mathbb{E}[f(\mathcal{V}^{u(1-\beta)} \cup B(0,k)^c)] \leq \eta(u) + \epsilon/2.
\end{equation}

Note that the local pullback $f_N^x$ of $f$ (see Definition~\ref{d:pull}) is given in this case by $f^x_N = 1\{x \text{ is connected to } \partial_i B(x,N^\delta/2) \text{ through } {\mathbb{T}\setminus X_{[0,uN^d]}}\}$. Using Proposition~\ref{p:average}, we conclude that
\begin{equation}
\label{e:fdens}
\lim\limits_{N \to \infty}P\Big[ \big| \bar f^u_N - \eta(u)\big| > \epsilon\Big] = 0,
\end{equation}
where $\bar f^u_N = \frac{1}{N^d}\sum_{x \in \mathbb{T}} f^x_N$. Note that this holds for any $\delta \in (0,1)$.

To finish the proof of \eqref{e:density2}, we choose $\delta = 1/4$ and observe that for $N$ large enough, at least one of the following possibilities must occur,
\begin{enumerate}
\item $\diam(\mathcal{C}^u_{max}) < N^{2\delta}$,
\item or $\mathcal{C}^u_{max} = \{y \in \mathbb{T}; y \text{ is connected to } \partial_i B(y,N^\delta/2) \text{ through } {\mathbb{T}\setminus X_{[0,uN^d]}}\}.$
\item or there is a $y \in \mathbb{T}$ which does not belong to $\mathcal{C}^u_{max}$ but is connected to $\partial_i B(y,N^\delta/2)$ through ${\mathbb{T}\setminus X_{[0,uN^d]}}$,
\end{enumerate}
Moreover, we note that in case 2, $|\mathcal{C}^u_{max}| = \bar f^u_N N^d$. Therefore, for $N$ large enough,
\begin{equation*}
\begin{split}
P\Big[& \frac{|\mathcal{C}^u_{max}|}{N^d} \not \in (\eta(u) - \epsilon, \eta(u) + \epsilon) \Big] \negmedspace \leq  P[\diam(\mathcal{C}^u_{max}) < N^{2\delta}] + P\Big[\bar f^u \not \in (\eta(u) - \epsilon, \eta(u)+\epsilon)\Big]\\
+ & P\Big[\text{there is a $y \in \mathbb{T}$ connected to $\partial_i B(y,N^\delta/2)$ through $\mathbb{T}\setminus X_{[0,uN^d]}$ but $y \not \in \mathcal{C}^u_{max}$}\Big],
\end{split}
\end{equation*}
and the three terms above converge to zero, due to Proposition~\ref{p:sec} and \eqref{e:fdens} (applied for $3\delta$ and $\delta$). This finishes the proof of Proposition~\ref{p:density} (hence the proof of Theorem~\ref{t:super}).
\end{proof}

\begin{remark}
1) It is an important question whether Theorem~\ref{t:d3} can be extended to all $u < u_\star$. This, together with Theorem~\ref{t:torus} would establish the existence of a sharp phase transition for the connectivity of $\mathbb{T} \setminus X_{[0,uN^d]}$ with respect to the intensity parameter $u$. Note that, using Theorems~\ref{t:torus} and \ref{t:super}, a much more precise statement would be obtained if one could show that $u_{\star \star} = u_\star$ and that every $u < u_\star$ is strongly supercritical in the sense of Definition~\ref{d:strong}. Hence, better understanding of random interlacements could directly establish the existence of a sharp phase transition in the component sizes for random walk on the torus.

2) Using the continuity of the function $\eta(u)$ in $[0,u_*)$, together with Proposition~\ref{p:average}, one can establish that for every $\epsilon > 0$,
\begin{equation}
\frac{1}{N^d}\big| \big\{\text{$x \in \mathbb{T}$; $\diam(\mathcal{C}^u_x) \geq N^{1-\epsilon}$}\big\}\big|  \;\; \xrightarrow[N \to \infty]{\text{in $P$-probability}} \eta(u), \text{ for every $u \neq u_*$}.
\end{equation}
This can be understood as a mesoscopic counterpart for the conjectured phase transition.
\end{remark}

\section{Preliminaries}
\label{s:prel}

The remainder of this article is devoted to the proof of Theorem~\ref{t:dom}. In this section, we collect some results on the expected hitting time of small subsets of $\mathbb{T}$ and on the quasistationary distribution, which will be important in the proof of Theorem~\ref{t:dom}. First we need to introduce some further notation.

\medskip

Recall that $P_x$ denotes the law of the simple random walk on $\mathbb T$ started at $x \in {\mathbb T}$, and $(X_n)_{n \geq 0}$ the canonical coordinate process. We now also introduce an independent Poisson point process $(N_t)_{t \geq 0}$ on $[0,\infty)$ with intensity $1$ and define the continuous-time random walk $Y_t = X_{N_t}$, $t \geq 0$. We can then view $(Y_t)_{t \geq 0}$ as an element in the space of cadlag functions from $[0,\infty)$ to $\mathbb T$ with the canonical $\sigma$-algebra generated by the coordinate projections, and introduce the canonical time-shift operators $(\theta_t)_{t \geq 0}$, such that $Y_t \circ \theta_s = Y_{t+s}$, for $s,t \geq 0$. For simplicity of notation, we also use $P_x$, $(X_n)_{n \geq 0}$, $(Y_t)_{t \geq 0}$ and $(\theta_t)_{t \geq 0}$ to denote the corresponding objects with $\mathbb T$ replaced by ${\mathbb Z}^d$. For any distribution $\mu$ on $\mathbb T$, we write $P_\mu$ for the law of the simple random walk on $\mathbb T$ with starting distribution $\mu$, meaning $P_\mu = \sum_{x \in {\mathbb T}_N} \mu(x) P_x$. For $\mu$ given by the uniform distribution $\pi$ on $\mathbb T$, we simply write $P$ rather than $P_\pi$, as before.

\medskip

The successive jump times of the continuous-time random walk are $\tau_n = \inf \{t \geq 0: N_t = n\},$ $n \geq 0$. The entrance and hitting times $H_V$ and ${\tilde H}_V$ of a set $V$ of vertices in $\mathbb T$ or ${\mathbb Z}^d$ are defined by
\begin{equation}
\label{e:H}
H_V = \inf \{t \geq 0: Y_t \in V\}, \quad
{\tilde H}_V = H_V \circ \theta_{\tau_1} + \tau_1,
\end{equation}
while the exit time of a set $V$ is defined as
\begin{align}
\label{e:T}
T_V = \inf \{t \geq 0: Y_t \notin V\}.
\end{align}
For ${\mathsf V} \subset {\mathbb Z}^d$, we define the equilibrium measure and capacity associated to $\mathsf V$ by
\begin{align}
\label{d:cap}
e_{\mathsf V}(x) = \mathbf{1}_{x \in {\mathsf V}} P_x[{\tilde H}_{\mathsf V} = \infty], \text{ for } x \in {\mathbb Z}^d, \quad  \capacity ({\mathsf V}) = e_{\mathsf V}({\mathbb Z}^d),
\end{align}
and for $x \in V \subseteq B(0,N/4)$, ${\mathsf V} = \phi(V)$, we define
\begin{align}
\label{e:eqfin}
 e_V(x) = e_{{\mathsf V}}(\phi(x)).
\end{align}
The trajectories of the continuous- and discrete-time random walks until time $t \geq 0$ are denoted $Y_{[0,t]} = \{Y_s, 0 \leq s \leq t\}$ and $X_{[0,t]} = \{X_n, 0 \leq n \leq t\}$. Note that in general we do not have $X_{[0,t]} = Y_{[0,t]}$, because $Y$ makes a random number of steps until time $t$. 
For any function $f: {\mathbb T} \to {\mathbb R}$, the Dirichlet form is defined by
\begin{align*}
{\mathcal D}(f,f) = \frac{1}{2} \sum_{x \in {\mathbb T} } \sum_{y \in {\mathbb T}: y \sim x} (f(x)-f(y))^2 \frac{1}{2d N^d},
\end{align*}
and related to the spectral gap $1-\lambda_2$ of $\mathbb T$ via
\begin{align}
  \label{def:gap}  
  1-\lambda_2 = \min \bigl\{ {\mathcal D}(f,f): \pi(f^2)=1, \pi(f)=0 \bigr\}, 
\end{align}
where $\pi(f) = \sum_{x \in {\mathbb T} } N^{-d} f(x)$. We define the regeneration time  
\begin{align}
\label{e:reg}
\reg = (N \log N)^2.
\end{align}
The following well-known estimate relates the regeneration time to convergence to equilibrium of the random walk (we refer to \cite{SC97}, p.~328, for a proof, and to Remark~2.2 in \cite{W08} for the fact that $1-\lambda_2 \geq cN^{-2}$; recall our convention on constants from the end of the introduction):
\begin{align}
 \label{eq:I} \sup_{x,y \in {\mathbb T}} |P_x[Y_\reg =y] - N^{-d}| \leq e^{-c \log^2 N}.
\end{align}

\input{random_inter}

%% file: random_inter.tex
Finally, we give a characterization of the law of the interlacement inside a given box $B(0,r) \subset \mathbb{Z}^d$, $r \geq 1$. We construct on some auxiliary probability space $(\Omega_{\ff B}, \mathcal{F}_{\ff B}, P_{\ff B}^u)$,
\begin{align}
&
\begin{array}{l}
\text{an iid sequence $X^i$, $i \geq 1$, of discrete time random walks,} \\
\text{starting with distribution $P_{e_{B(0,r)}}/\text{cap}(B(0,r))$, and}
\end{array}\\
& \begin{array}{l} \text{an independent Poisson random variable $J$ with intensity $u \,\text{cap}(B(0,r))$.}\end{array}
\end{align}

With this we can state the following characterization, which will be used in the proof of Theorem~\ref{t:dom}. Recalling that $\mathcal{I}^u$ stands for the interlacement set at level $u$ under $\mathbb{P}$,
\begin{equation}
\label{e:interlac}
\begin{array}{c}
\text{$\mathcal{I}^u \cap B(0,r)$, has the same distribution as} \\
\text{$\bigcup_{i \leq J} \text{range}(X^i) \cap B(0,r)$ under the law $P_{\ff B}^u$,}
\end{array}
\end{equation}
see for instance, \cite{Szn09}, Proposition~1.3 and below (1.42).

%% file: preliminaries.tex
\subsection{Expected entrance time}

We now collect some preliminary estimates and deduce a first key statement in Proposition~\ref{p:Gloc}. This proposition proves that for suitably small sets $V \subset B(0,N/4)$ and ${\mathsf V} = \phi(V)$, $N^d/E[H_V]$ is close to $\capacity({\mathsf V})$. This statement may well be known, but we could not find a proof in the literature.

\medskip

For the rest of this article, we fix any $\epsilon \in (0,1),$
we consider any 
\begin{align}
\label{e:rN}
 1 \leq r_N \leq N^{1-\epsilon},
\end{align}
and define the concentric boxes
\begin{align}
\label{e:boxes}
&A = B(0, r_N) \subseteq B = B(0, N^{1- \epsilon/2}) \subseteq C = B(0, N/4) \subseteq {\mathbb T},
\end{align}
as well as ${\mathsf A} = \phi(A)$, ${\mathsf B} = \phi(B)$ and $\mathsf{C}=\phi(C)$. To begin with, we collect some elementary bounds on hitting probabilities. The following lemma asserts in particular that the random walk on $\mathbb T$ typically exits $C$ before entering $A$ when started outside of $B$ \eqref{a:gotoC}, and typically does not enter $A \cup \partial_e A$ ($B \cup \partial_e B$) before time $\reg$ from outside of $B$ ($C$, respectively, \eqref{a:Drat}, \eqref{a:Drat'}). 

\begin{lemma} \label{l:basic}
\textup{($d \geq 3$)}
For $B' = B(0,N^{1-\epsilon/2}/2)$,
\begin{align}
\label{a:gotoC} \sup_{x \in {\mathbb T} \setminus B'} P_x [H_A \leq T_{C}] &\leq \Rbd,\\
\label{a:Drat} \sup_{x \in {\mathbb T} \setminus B} P_x [H_{A \cup \partial_e A} \leq \reg] &\leq \rbd, \\
\label{a:Drat'} \sup_{x \in {\mathbb T} \setminus C} P_x [H_{B \cup \partial_e B} \leq \reg] &\leq \rbdd.
\end{align}
\end{lemma}

\begin{proof}
The statements follow from classical random walk estimates, therefore we postpone their proofs to the Appendix.
\end{proof}

The next lemma collects basic facts on escape probabilities and capacities.

\begin{lemma}
\textup{($d \geq 3$)}
\begin{align}
\label{a:escest} c_\epsilon N^{\epsilon-1} &\leq  \inf_{x \in \partial_i {\mathsf A}} e_{\mathsf A}(x),\\
 \label{a:capest} 
c_\epsilon N^{(1-\epsilon)(d-2)} &\leq \capacity ({\mathsf A}) \leq c'_\epsilon N^{(1-\epsilon)(d-2)}, \\
 \label{a:capestB}
c_\epsilon N^{(1-\epsilon/2)(d-2)} &\leq \capacity ({\mathsf B}) \leq c'_\epsilon N^{(1-\epsilon/2)(d-2)}.
\end{align}
\end{lemma}

\begin{proof}
A standard estimate on one-dimensional simple random walk, see \cite{D05}, Chapter~3, Example~1.5, p.~179, implies that $\inf_{x \in \partial_i {\mathsf A}} P_x [T_{B(0,2N^{1-\epsilon})} < {\tilde H}_{\mathsf A}] \geq c_\epsilon N^{\epsilon-1}$, so \eqref{a:escest} follows from \cite{L91}, Proposition 1.5.10 and the strong Markov property applied at time $T_{B(0,2N^{1-\epsilon})}$.
The proofs of \eqref{a:capest} and \eqref{a:capestB} are contained in \cite{L91}, Proposition~2.2.1 (a) and (2.16), p.~52 and 53.
\end{proof}

The following proposition, quoted from \cite{CTW}, will be instrumental in relating expected entrance times to capacities. The statement essentially results from the finite graph analogue of the Dirichlet principle and asserts that the Dirichlet form of the so-called equilibrium potential with respect to disjoint subsets $A_1$ and $A_2^c$ \eqref{e:potential} can be used to approximate the reciprocal of the expected entrance time of $A_1$.

\begin{proposition} \label{pr:EH}
\textup{($d \geq 3$)}
Let $A_1 \subseteq A_2 \subseteq {\mathbb T}$, and let $f$ and $g: {\mathbb T} \to {\mathbb R}$ be defined by
\begin{align}
\label{e:potential} g (x) = P_x[H_{A_1} \leq T_{A_2}], \text{ for } x \in {\mathbb T}, \text{ and}
\end{align}
\begin{align}
\label{e:fstar} f_{A_1} (x) = 1 - \frac{E_x[H_{A_1}]}{E[H_{A_1}]}, \text{ for } x \in {\mathbb T}.
\end{align}
Then
\begin{equation}
  \label{e:EH} 
  {\mathcal D}(g,g) \Bigl( 1 - 2 \sup_{x \in {\mathbb T} \setminus A_2} |f_{A_1}(x)| \Bigr) \leq
  \frac{1}{E[H_{A_1}]} \leq {\mathcal D}(g,g) \frac{1}{\pi({\mathbb T} \setminus A_2)^2}.
\end{equation}
\end{proposition}

\begin{proof}
See \cite{CTW}, Proposition~3.1.
\end{proof}

\begin{remark} \label{r:Dir}
A simple computation shows that, for $g$ as in \eqref{e:potential},
\begin{align}
\label{e:Dgg} {\mathcal D}(g, g) = \sum_{x \in A_1} P_x [T_{A_2} < {\tilde H}_{A_1} ] \frac{1}{N^d},
\end{align}
which will be useful in the sequel.
\end{remark}

First, we apply the right-hand estimate in \eqref{e:EH}, to obtain an estimate which is probably known, but does not seem to be proved anywhere in the literature.

\begin{lemma} \label{l:Hlbd}
\textup{($d \geq 3$)}
For any $V \subseteq B$,
  \begin{equation}
    \label{e:Hlbd} 
    \frac{1}{E[H_V]} \leq (1+ c_\epsilon \rbdd) \frac{\capacity({\mathsf V})}{N^d}. 
  \end{equation}
\end{lemma}

\begin{proof}
We apply the right-hand estimate in \eqref{e:EH} with $A_1 = V$ and $A_2 = B(0,N^{1-\epsilon/4})$, and Remark~\ref{r:Dir}: 
\begin{align*}
\frac{1}{E[H_V]} &\leq (1+ cN^{-d \epsilon/4})  {\mathcal D}(g, g) =(1+ cN^{-d \epsilon/4}) \sum_{x \in \partial_i V} P_x [T_{A_2}<{\tilde H}_V] \frac{1}{N^d} \\
&= (1+ cN^{-d \epsilon/4}) \sum_{x \in \partial_i {\mathsf V}} \bigl( P_x [{\tilde H}_{\mathsf V} = \infty] + P_x [T_{{\mathsf {A_2}}}<{\tilde H}_{\mathsf V} , H_{\mathsf V} < \infty ] \bigr) \frac{1}{N^d},
\end{align*}
where we have used the isomorphism $\phi$ in the last step.
Applying the strong Markov property at time $T_{\mathsf B}$, we obtain for any $x \in \partial_i {\mathsf V}$,
\begin{equation}
\label{Hlbd1}
\begin{split}
 P_x [ T_{\mathsf {A_2}} < {\tilde H}_{\mathsf V}, H_{\mathsf V} < \infty] \leq P_x [T_{\mathsf {A_2}} < {\tilde H}_{\mathsf V}] \sup_{y \in \partial_e {\mathsf {A_2}}} P_y [ H_{\mathsf B} < \infty ].
\end{split}
\end{equation}
By \cite{L91}, Proposition 1.5.10, p.~36, we have for any $y \in \partial_e {\mathsf {A_2}}$, $P_y [ H_{\mathsf B} < \infty ] \leq N^{-c_\epsilon}$ and thus also $P_y [ H_{\mathsf B} = \infty ] \geq c_\epsilon >0$. In particular, we deduce from \eqref{Hlbd1} that 
\begin{align*}
 P_x [ & T_{\mathsf {A_2}} < {\tilde H}_{\mathsf V} , H_{\mathsf V} < \infty] \leq \Rbd P_x [ T_{\mathsf {A_2}} < {\tilde H}_{\mathsf V}] \\
& \leq  \Rbd P_x [ T_{\mathsf {A_2}} < {\tilde H}_{\mathsf V}] \inf_{y \in \partial_e {\mathsf {A_2}}} P_y [ H_{\mathsf B} = \infty ] / c_\epsilon \leq \rbdd P_x [{\tilde H}_{\mathsf V} = \infty] / c_\epsilon ,
\end{align*}
which with the first estimate in this proof yields \eqref{e:Hlbd}.
\end{proof}

We now control the function $f_V$ appearing on the left-hand side of \eqref{e:EH}, which essentially amounts to showing that the precise location of the starting point does not matter for expected entrance times of subsets $V$ of $A$, provided the random walk starts outside of $B$. The idea is that, due to \eqref{a:Drat}, the random walk typically does not enter $A$ until well after the regeneration time, at which time the distribution of the walk is close to uniform regardless of the starting point.

\begin{proposition} \label{pr:Eratio}
\textup{($d \geq 3$)}
  For any $V \subseteq A$ and $f_V$ defined in \eqref{e:fstar}, 
  \begin{align}
    \label{e:flbd} \inf_{x \in {\mathbb T}} f_V (x) &\geq - c_\epsilon N^{-c_\epsilon},\\
    \label{e:fbd} \sup_{x \in {\mathbb T} \setminus B} |f_V (x)| &\leq c_\epsilon \rbd.
    \end{align}
\end{proposition}

\begin{proof}
Let us first consider the expectation of $H_V$ when starting from $Y_{\reg}$. From \eqref{eq:I} we obtain, for any $x \in {\mathbb T}$,
  \begin{equation}
    \begin{split}
      \label{Er2} 
      \bigl| E_x[ E_{Y_{\reg}}[H_V] ] - E[H_V] \bigr| \leq \sum_{y \in {\mathbb T}} |P_x[Y_\reg=y] - N^{-d}| \sup_{z \in {\mathbb T}} E_z [H_V] \le e^{-c \log^2 N}, 
    \end{split}
  \end{equation}
  where we have bounded the expected entrance time of $V$ by $cN^d$ (see, for example, \cite{LPW09}, Proposition~10.13, p.~133).
  We now apply this inequality to find an upper bound on $E_x[H_V]$. Since $H_V \leq \reg + H_{V} \circ \theta_{\reg}$, the simple Markov property applied at time $\reg$ and \eqref{Er2} imply that 
  \begin{equation}
    \label{Er3} \sup_{x \in {\mathbb T}} E_x[H_{V}] \leq \reg + e^{-c\log^2 N} + E[H_{V}]. 
  \end{equation}
  With \eqref{e:Hlbd} and \eqref{a:capest} (as well as monotonicity of $\capacity(.)$, see for instance Proposition~2.3.4 $(a)$ of \cite{L91}), we deduce that
  \begin{align}
\label{Er4}
     \sup_{x \in {\mathbb T}} \frac{E_x[H_{V}]}{E[H_{V}]} - 1 &\leq  (\reg + e^{-c\log^2 N})   c_\epsilon \frac{\capacity({\mathsf V})}{N^d} \leq (\reg + e^{-c\log^2 N})  c_\epsilon N^{-2-\epsilon (d-2)}. 
  \end{align}
Since $\reg = N^2 \log^2 N$ and $d \geq 3$,  this proves \eqref{e:flbd}.

  We now consider any $x \in {\mathbb T} \setminus B$. By the simple Markov property applied at time $\reg$, 
\begin{equation*}
    \begin{split}
      E_x[H_{V}] 
      &\geq E_x [\mathbf{1}_{\{H_{B} > \reg\}} E_{Y_{\reg}}[H_{V}] ] = E_x [ E_{Y_{\reg}}[H_{V}] ] - E_x [\mathbf{1}_{\{H_{B} \leq \reg\}} E_{Y_{\reg}}[H_{V}] ] \\
      &\stackrel{\eqref{Er2}}{\geq} E[H_{V}] - e^{-c\log^2 N} -
      P_x[H_{B} \leq \reg]  \sup_{y \in {\mathbb T}} E_y[H_{V}]\\
      &\stackrel{\eqref{Er3}}{\geq} E[H_{V}] - 2e^{-c\log^2 N} -
      P_x[H_{B} \leq \reg]  ( \reg + E[H_{V}]).
    \end{split}
  \end{equation*}
  With \eqref{a:Drat}, \eqref{a:capest} and  \eqref{e:Hlbd}, this yields
\begin{equation*}
    \inf_{x \in {\mathbb T} \setminus B} \frac{E_x[H_{V}]}{E[H_{V}]} -1 \geq - 2e^{-c\log^2 N} -  \rbd \Bigl( c_\epsilon (\log N)^2 N^{-\epsilon(d-2)} + 1 \Bigr).
  \end{equation*}
Together with \eqref{e:flbd}, this proves \eqref{e:fbd}.
\end{proof}

Finally, we combine the above estimates to exhibit the link between expected entrance times and capacities. 

\begin{proposition} \label{p:Gloc}
\textup{($d \geq 3$)}
For any $V \subseteq A$,
 \begin{align}
  \label{e:Gloc}  \biggl| \frac{N^d}{E[H_{V}] \capacity({\mathsf V})} - 1 \biggr| \leq c_\epsilon \rbd.
 \end{align}
\end{proposition}

\begin{proof}
We use $g^*$ to denote the function defined in \eqref{e:potential} with $A_1 = V$ and $A_2 = B$. Let us first compare the effective conductance ${\mathcal D}(g^*,g^*)$ between $V$ and ${\mathbb T} \setminus B$ with the capacity of the set ${\mathsf V}$. In view of \eqref{e:Dgg}, $ \Bigl| N^d {\mathcal D}(g^*,g^*) - \capacity ({\mathsf V}) \Bigr|$ is equal to
\begin{align*}
 \biggl| \sum_{x \in \partial_i V} \Bigl( P_x [T_{B} < {\tilde H}_{V}] - P_{\phi(x)} [ {\tilde H}_{{\mathsf V}} = \infty ] \Bigr)  \biggr| = \sum_{x \in \partial_i {\mathsf V}} P_x [T_{{\mathsf B}} < {\tilde H}_{{\mathsf V}}, H_{{\mathsf V}} < \infty] .
\end{align*}
With the strong Markov property applied at time $T_{{\mathsf B}}$ and the same argument as below \eqref{Hlbd1}, it follows that 
\begin{align}
\label{loc1.2} \Bigl| N^d {\mathcal D}(g^*,g^*) - \capacity ({\mathsf V}) \Bigr| &\leq   c_\epsilon \Rbd \capacity ({\mathsf V}).
\end{align}
We now use this estimate in the right-hand inequality in \eqref{e:EH} and obtain
\begin{align}
\label{loc1.3} \frac{N^d}{E[H_{V}] \capacity ({\mathsf V})} \leq 1 +  c_\epsilon \Rbd.
\end{align}
On the other hand, applying \eqref{loc1.2} to the left-hand inequality in \eqref{e:EH}, we have
\begin{align*}
 \frac{N^d}{E[H_{V}] \capacity ({\mathsf V})} \geq (1- c_\epsilon \Rbd) \bigl(1 - 2 \sup_{x \in {\mathbb T} \setminus B} |f^*_{V}(x)| \bigr) .
\end{align*}
Together with \eqref{e:fbd} and \eqref{loc1.3}, this proves Proposition~\ref{p:Gloc}.
\end{proof}

The following is a discrete version of the Kac moment formula, also known as \\Kha\'sminskii's Lemma (cf.~\cite{K59}):

\begin{lemma} \label{l:kac}
\textup{($d \geq 3$)}
 For any $V \subseteq {\mathbb T}$, $x \in {\mathbb T}$ and $k \geq 1$,
\begin{align}
 \label{e:kac} E_x[H_V^k] \leq k \textup{!} \sup_{y \in {\mathbb T}} E_y [H_V]^k.
\end{align}
\end{lemma}

\begin{proof}[Proof of Lemma~\ref{l:kac}.]
See \cite{FP99}, equation (4) and the relevant special case (6). 
\end{proof}

%% file: quasi.tex
\subsection{The quasistationary distribution} \label{s:quasi}

We now introduce the quasistationary distribution on ${\mathbb T} \setminus B$ and some of its key properties. The importance of the quasistationary distribution is highlighted by Lemma~\ref{l:quasi}, showing that it is characterized as the equilibrium distribution of the random walk conditioned not to enter $B$. This fact will later allow us to show approximate independence between appropriately defined sections of the random walk trajectories and thereby make the random interlacements appear.

\medskip

In order to define the quasistationary distribution, we consider, for $B$ as in \eqref{e:boxes}, the $(N-|B|) \times (N-|B|)$-matrix 
\begin{align}
\label{e:transB}
P^B = \Big( \frac{1}{2d} \mathbf{1}_{\{x \sim y\}} \Big)_{x,y \in {\mathbb T} \setminus B}.
\end{align}
By the Perron-Frobenius theorem, the symmetric and irreducible matrix $P^B$ has a unique largest eigenvalue $\lambda^B_1$, whose associated eigenvector $v_1$ has non-negative entries (see \cite{S02}, Theorem~5.3.1, p.~82). The quasistationary distribution $\sigma$ on ${\mathbb T} \setminus B$ is then defined by
\begin{align}
\label{d:quasi}
\sigma(x) = \frac{(v_1)_x}{v_1^T \mathbf{1}},
\end{align}
where $(v_1)_x$ denotes the $x$-entry of the column vector $v_1$, and $\mathbf{1}$ denotes the vector with all entries equal to $1$. 
We now come to the key lemma, showing that the distribution of the random walk at time $\reg$ conditioned not to have entered $B$ is close to the quasistationary distribution.

\begin{lemma}
\textup{($d \geq 3$)}
\label{l:quasi}
\begin{align}
\label{e:quasi}
&\sup_{x, y \in {\mathbb T} \setminus B} \left| P_x[Y_\reg = y| H_B > \reg] - \sigma(y) \right| \leq e^{-c_\epsilon \log^2 N}.
\end{align}
\end{lemma} 

\begin{proof}
Although we expected to find a proof of this lemma in the literature, we did not. A complete proof is given in the Appendix.
\end{proof}

Finally, we prove that the hitting distribution of $A$ by the random walk started from the quasistationary distribution with respect to $B$ is close to the normalized equilibrium measure on $A$. Together with the previous lemma, this shows in particular that successive visits to the set $A$ by the random walk, when separated by time intervals of length $\reg$ in which the walk is conditioned not to have hit $B$, are close to independent.

\begin{lemma} \label{l:quni}
\textup{($d \geq 3$)}
\begin{equation}
 \label{e:quni}
 \sup_{x \in \partial_i A} \left| \frac{P_\sigma [Y_{H_A}=x] \capacity({\mathsf A})}{e_A(x)} - 1 \right| \leq c_\epsilon \rbdd.
\end{equation}
\end{lemma}

\begin{proof}
Let us consider the probability that the random walk started at $x \in \partial_i A$ stays outside of $B$ for a time interval of length $\reg$ before returning to $A$, and then returns to $A$ through some vertex other than $x$. By reversibility of the random walk with respect to the uniform distribution on $\mathbb T$, this probability can be written as
\begin{align}
 \label{quni1}
\sum_{y \in \partial_i A \setminus \{x\}} P_x [ U< {\tilde H}_A , Y_{{\tilde H}_A} =y] = \sum_{y \in \partial_i A \setminus \{x\}} P_y [ U<{\tilde H}_A , Y_{{\tilde H}_A} =x],
\end{align}
where
\begin{align}
\label{d:U} U = \inf \{t \geq \reg: Y_{[t- \reg,t]} \cap B = \emptyset \}.
\end{align}
We now denote the step of the last visit to $B$ before $U$ as (cf.~the second paragraph of this section for the notation)
\begin{align}
\label{d:L} 
L = \sup \{0 \leq l \leq N_U: Y_{\tau_l} \in B\}.
\end{align}
Summing over all possible values of $L$ and $Y_{\tau_L}$, we have
\begin{align*}
 P_x [ U< {\tilde H}_A , Y_{{\tilde H}_A} =y] &= \sum_{l \geq 0, z \in \partial_i B} P_x [L=l, Y_{\tau_l} = z, U< {\tilde H}_A , Y_{{\tilde H}_A} =y] \\
&= \sum_{l \geq 0, z \in \partial_i B} P_x [Y_{\tau_l} = z, \tau_l < {\tilde H}_A \wedge U, H_B \circ \theta_{\tau_{l+1}} > \reg , Y_{{\tilde H}_A} =y].
\end{align*}
Applying the simple Markov property at the times $\tau_{l+1}$ and $\tau_{l+1} + \reg$, the probability on the right-hand side becomes
\begin{align*}
\sum_{x' \in {\mathbb T} \setminus B} E_x \Big[ Y_{\tau_l}=z, \tau_l < {\tilde H}_A \wedge U , P_{Y_{\tau_{l+1}}} [H_B> \reg]  P_{Y_{\tau_{l+1}}} [Y_\reg = x' | H_B> \reg]  \Big] P_{x'} [Y_{H_A}=y],
\end{align*}
hence by Lemma~\ref{l:quasi},
\begin{align*}
 \left| P_x [ U< {\tilde H}_A , Y_{{\tilde H}_A} =y] - P_x [U< {\tilde H}_A ] P_\sigma [Y_{H_A}=y] \right| \leq e^{-c_\epsilon \log^2 N}.
\end{align*}
Applying this estimate to both sides in \eqref{quni1}, we obtain
\begin{align*}
\bigg| P_x [U< {\tilde H}_A] P_\sigma [Y_{H_A} \neq x] -  P_\sigma [Y_{H_A}=x] \sum_{y \in \partial_i A \setminus \{x\}} P_y [U<{\tilde H}_A ] \bigg| \leq e^{-c_\epsilon \log^2 N},
\end{align*}
or equivalently,
\begin{align}
\label{quni2}
\bigg| P_x [U<{\tilde H}_A] -  P_\sigma [Y_{H_A}=x] \sum_{y \in \partial_i A} P_y [U<{\tilde H}_A ] \bigg| \leq e^{-c_\epsilon \log^2 N}.
\end{align}
For any $x \in \partial_i A$, we have by \eqref{a:Drat'} and the strong Markov property applied at time $T_C$,
\begin{equation}
 \label{quni3}
\begin{split}
 P_x [U< {\tilde H}_A ] &\geq P_x [T_C < {\tilde H}_A] \inf_{z \in {\mathbb T} \setminus C} P_z [H_B > \reg] \geq e_A(x) (1- \rbdd), \text{ cf.~\eqref{d:cap}.}
\end{split}
\end{equation}
On the other hand, $P_x [U<{\tilde H}_A]$ is bounded from above by
\begin{align*}
P_x & [T_B<{\tilde H}_A] = P_{\phi(x)} [{\tilde H}_{\mathsf A} = \infty] + P_{\phi(x)} [T_{\mathsf B} < {\tilde H}_{\mathsf A}, {\tilde H}_A < \infty] \\
&\leq P_{\phi(x)} [{\tilde H}_{\mathsf A} = \infty] + P_{\phi(x)} [ T_{\mathsf B} < {\tilde H}_{\mathsf A}] \sup_{z \in {\mathbb Z}^d \setminus {\mathsf B}} P_z [{\tilde H}_{\mathsf A} < \infty] \leq e_A(x) (1 + c_\epsilon \Rbd),
\end{align*}
by \eqref{a:gotoC}. Together with \eqref{quni3}, we obtain that for any $x \in \partial_i A$,
\begin{align*}
(1-\rbdd) e_A(x) \leq P_x [U<{\tilde H}_A] \leq (1+ c_\epsilon \Rbd) e_A(x),
\end{align*}
which implies that 
\begin{align}
\label{quni4} 
\bigg| \frac{P_x [U<{\tilde H}_A] \capacity({\mathsf A})}{\sum_{y \in \partial_i A} P_y [U<{\tilde H}_A] e_A(x)} - 1 \bigg| \leq  c_\epsilon \rbdd.
\end{align}
Since $e_A(x) \geq c_\epsilon N^{\epsilon-1}$ by \eqref{a:escest}, multiplication of \eqref{quni2} by $\frac{\capacity({\mathsf A})}{e_A(x)\sum\limits_{y \in \partial_i A} P_y [U<{\tilde H}_A]}$ yields
\begin{align*}
\bigg| \frac{P_x[U<{\tilde H}_A ] \capacity({\mathsf A})}{\sum_{y \in \partial_i A} P_y [U<{\tilde H}_A] e_A(x)} - \frac{P_\sigma [Y_{H_A}=x] \capacity ({\mathsf A})}{e_A(x)} \bigg| \leq e^{-c_\epsilon \log^2 N},
\end{align*}
and together with \eqref{quni4} completes the proof.
\end{proof}

%% file: poisson.tex
\section{Poissonization}
\label{s:dom}

We now come to the Poissonization step of the domination argument, culminating in Proposition~\ref{p:c2}. This proposition provides a coupling between the random walk trajectory and two Poisson random measures on the space $\Gamma$ of trajectories in $\mathbb T$, in such a way that the traces of these random measures dominate the random walk trajectory intersected with $A$ from above and from below with high probability. This coupling will then be a crucial part for the domination of $X(u, {\mathsf A})$ by random interlacements, carried out in Sections~\ref{s:dom+} and \ref{s:dom-}.

\medskip

We begin by chopping up the random walk into suitable excursions. In words, the random walk starts an excursion by entering $A$, and ends the excursion as soon as it has not visited $B$ for a time interval of length $\reg$ for $A, B$ defined in \eqref{e:boxes}, see Figure~\ref{f:timeU}. Formally, we recall the definition of $U$ from \eqref{d:U} and define the successive return and end times by (cf.~Figure~\ref{f:timeU})
\begin{equation}
\label{e:exc}
\begin{split}
&R_1 = H_A, \, U_1 = R_1 + U \circ \theta_{R_1} \,\text{ and for $k \geq 2$,}\\
&R_k = U_{k-1} + R_1 \circ \theta_{U_{k-1}}, \,  \, U_k = U_{k-1} + U_1 \circ \theta_{U_{k-1}}.
\end{split}
\end{equation}
\begin{figure}
\psfrag{A}[cl][cl][2][0]{$A$}
\psfrag{B}[cl][cl][2][0]{$B$}
\psfrag{R1}[cl][cl][2][0]{$R_1$}
\psfrag{R2}[cl][cl][2][0]{$R_2$}
\psfrag{t}[cl][cl][2][0]{$t_*$}
\psfrag{U}[cl][cl][2][0]{$U_1$}
\begin{center}
\includegraphics[angle=0, width=0.3\textwidth]{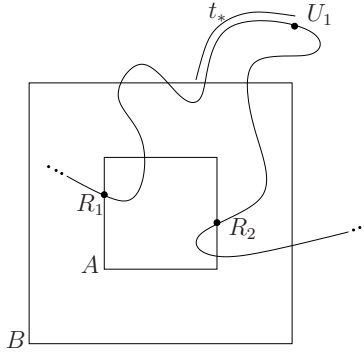}\\
\caption{The times defined in \eqref{e:exc}.}\label{f:timeU}
\end{center}
\end{figure}
The random walk trajectories between the times $R_i$ and $U_i$ will then be compared with independent trajectories. On an auxiliary probability space $({\bar \Omega}, {\bar {\mathcal F}}, {\bar P}_\sigma)$, we thus introduce
\begin{equation}
\label{e:Yhat}
\begin{split}
&\text{iid random walks $(\bar Y^i)_{i \geq 1}$, distributed as $(Y_{t \wedge U_1})_{t \geq 0}$ under $P_\sigma$,}
\end{split}
\end{equation}
as well as, for any $u>0$ and $\epsilon \in (0,1/3)$ that remain fixed throughout this section, 
\begin{equation}
\label{e:J}
\begin{split}
&\text{independent random variables $J^-$ and $J^+$ with Poisson}\\
&\text{distribution with parameters $(1-2 \epsilon) u \capacity ({\mathsf A})$ and $(1+2 \epsilon) u \capacity ({\mathsf A})$.} 
\end{split}
\end{equation}

We need a basic large deviations estimate on $J^\pm$:

\begin{lemma} \label{l:poissonld}
\textup{($d \geq 3$)}
\begin{align}
\label{e:poissonld}
{\bar P}_\sigma[ J^- \leq (1-3\epsilon/2) u \capacity ({\mathsf A}) \leq (1+3\epsilon/2) u \capacity ({\mathsf A}) \leq J^+ ] \geq 1 - e^{-c_{u,\epsilon} \capacity({\mathsf A})}. 
\end{align}
\end{lemma}

\begin{proof}
The statement follows from a standard exponential bound on the probability that the Poisson-distributed random variable $J^\pm$ does not take a value in the interval $(1 \pm 2 \epsilon - \epsilon/2) u \capacity({\mathsf A}), (1 \pm 2 \epsilon + \epsilon/2) u \capacity({\mathsf A}) )$.
\end{proof}

The estimates derived in the previous section now allow us to relate in the following lemma the (dependent) random walk excursions $Y_{[R_i,U_i]}$ to the independent excursions ${\bar Y}^i_{[R_i,U_i]}$. Note that the first excursion $Y_{[R_1,U_1]}$ does not feature in the statement. The reason is that the uniformly chosen starting point of the random walk makes $Y_{R_1}$ behave differently from the other entrance points in $A$.

\begin{lemma}
\label{l:decouple}
\textup{($d \geq 3$)}
For any $k \geq 2$, there exists a coupling $(\Omega_0, {\mathcal F}_0, Q_0)$ of \\ $\left( Y_{[R_i,U_i]} \cap A \right)_{i=2}^k$ under $P$ and $\big( \bar Y^{i}_{[R_1,U_1]} \cap A \big)_{i=2}^k$ under ${\bar P}_\sigma$, such that
\begin{align}
\label{e:decouple}
Q_0 \left[ \left( Y_{[R_i,U_i]} \cap A \right)_{i=2}^k = \left( \bar Y^{i}_{[R_1,U_1]} \cap A \right)_{i=2}^k \right] \geq 1- k e^{-c_\epsilon \log^2 N}.
\end{align}
\end{lemma}

\begin{proof}
For each $x \in {\mathbb T} \setminus B$, we use Lemma~\ref{l:quasi} and \cite{LPW09}, Proposition~4.7, p.~50, to construct a coupling $q_x$ of $Y_\reg$ under $P_x [.|H_B>\reg]$ and a $\sigma$-distributed random variable $\Sigma$ such that
\begin{align}
\label{dec00}
q_x [Y_\reg \neq \Sigma] \leq e^{-c_\epsilon \log^2 N}.
\end{align}
For $L$ as in \eqref{d:L} and $i \geq 1$, we define $L_i = L \circ \theta_{R_i} + N_{R_i}$ as the last step at which the $i$-th excursion is in $B$.
For simplicity, we write
\begin{align*}
{\mathcal A}_i = Y_{[R_i,U_i]} \cap A = Y_{[R_i,\tau_{L_i}]} \cap A, \text{ and } {\bar {\mathcal A}}_i = {\bar Y}^i_{[R_1,U_1]} \cap A = {\bar Y}^i_{[R_1,\tau_{L_1}]} \cap A, 
\end{align*}
as well as ${\mathcal A} =  \left( {\mathcal A}_i  \right)_{i=2}^k$  and $\bar{\mathcal A} = \left( {\bar {\mathcal A}}_i\right)_{i=2}^k$ 
throughout this proof. In particular, our task is to construct a coupling of $\mathcal A$ and $\bar{\mathcal A}$.
We use the coupling in \eqref{dec00} to couple $\mathcal A$ and $\bar{\mathcal A}$ together with two $({\mathbb T} \setminus B \times \partial_e B)^{k-1}$-valued random variables ${\mathcal X}$ and ${\bar {\mathcal X}}$, distributed as $(Y_{U_{i-1}},X_{L_i+1})_{i=2}^k$ under $P$ and as $(Y_0^i,X^i_{L_1+1})_{i=2}^k$ under ${\bar P}_\sigma$. In words, the construction goes as follows: given any $x_1^+ \in \partial_e B$ chosen according to $P[X_{L_1 +1}=\cdot]$, we choose $x_2$ and ${\bar x}_2 \in {\mathbb T} \setminus B$ according to $q_{x_1^+}[Y_{t_*}=\cdot, \Sigma = \cdot]$. If $x_2$ and ${\bar x}_2$ are equal (which is the typical case, cf.~\eqref{dec00}), then we choose $S_2 = \bar{S}_2 \in 2^A$ and $x_2^+ = {\bar x}_2^+ \in \partial_e B$ according to $P_{x_2}[{\mathcal A}_1 = \cdot, X_{L_1 +1}=\cdot]$. If $x_2$ and ${\bar x}_2$ differ, then we choose $(S_2,x_2^+)$ and $(\bar{S}_2,{\bar x}_2^+)$ independently according to $P_{x_2}[{\mathcal A}_1 = \cdot, X_{L_1+1}=\cdot]$ and $P_{{\bar x}_2}[{\mathcal A}_1 = \cdot X_{L_1 +1}=\cdot]$. In any case, we repeat the above with $x_2^+$ in place of $x_1^+$ and iterate until step $k$. Formally, for $S = (S_2, \ldots, S_k)$ and $\bar{S} = (\bar{S}_2, \ldots, \bar{S}_k) \in (2^A)^{k-1}$, and $\mathbf{x} = (x_2, x_2^+, \ldots, x_k, x_k^+)$ and $\bar{\mathbf{x}} = ({\bar x}_2, {\bar x}_2^+, \ldots, {\bar x}_k , {\bar x}_k^+ ) \in ({\mathbb T} \setminus B \times \partial_e B)^{k-1}$, we set
\begin{equation}
\label{dec0}
\begin{split}
&Q_0 \left[{\mathcal A} = S, {\mathcal X} = \mathbf{x} , \bar{\mathcal A} = \bar{S}, \bar{\mathcal{X}} = \bar{\mathbf{x}} \right]\\
&= \sum_{x_1^+ \in \partial_e B} P[X_{L_1+1} = x_1^+] \prod_{i=2}^k \bigg( q_{x_{i-1}^+} \left[ Y_{t_*} = x_i, \Sigma = {\bar x}_i \right]\\
&\qquad \qquad \qquad \Big( \mathbf{1}_{x_i = \bar{x}_i} P_{x_i} [{\mathcal A}_1 = S_i, X_{L_1 +1} = x_i^+ ] \mathbf{1}_{x_i^+ = \bar{x}_i^+, S_i = \bar{S}_i}\\
&\qquad \qquad \qquad + \mathbf{1}_{x_i \neq \bar{x}_i} P_{x_i} [{\mathcal A}_1 = S_i, X_{L_1 +1} = x_i^+ ] P_{{\bar x}_i} [{\mathcal A}_1 = \bar{S}_i, X_{L_1 +1} = {\bar x}_i^+ ] \Big) \bigg) .
\end{split}
\end{equation}
Let us check that $\mathcal A$ and $\bar{\mathcal A}$ indeed have the claimed distributions under $Q_0$. 
Summing \eqref{dec0} over $S$ and $\mathbf x$, one obtains 
\begin{equation*}
\begin{split}
Q_0[\bar{\mathcal A} = \bar{S}, \bar{\mathcal X} = \bar{ \mathbf x}] &= \prod_{i=2}^k \sigma(\bar{x}_i)P_{\bar{x}_i} [\bar{\mathcal A}_1 = \bar{S}_i, X_{L_1+1}= \bar{x}_i^+]\\
&= \bar{P}_\sigma \left[ \bar{\mathcal A} = \bar{S}, (Y^i_0,X^i_{L_1+1})_{i=2}^k = \bar{\mathbf x} \right],
\end{split}
\end{equation*}
which upon summation over $\bar{\mathbf x}$ yields $Q_0 [\bar{\mathcal A} = \bar{S}] = \bar{P}_\sigma [\bar{\mathcal A}=\bar{S}]$, as required.
On the other hand, observe that, although $L_1$ is not a stopping time, we have $\{X_l \in B, L_1 \geq l\} \in {\mathcal F}_{\tau_l}$, and that $\{L_1 = l\} = \{X_l \in B, L_1 \geq l\} \cap \theta^{-1}_{\tau_l} \{H_B > t_*\}$ for $l \geq 0$. Hence, the Markov property shows that for any $2 \leq i \leq k$ and any $x \in {\mathbb T}$ and $S' \subseteq A$,
\begin{equation}
\label{dec1}
\begin{split}
&P_x [{\mathcal A}_1 = S', X_{L_1+1}= x_i^+ ] q_{x_i^+} [Y_\reg=x_{i+1}] \\
&\qquad = \sum_{l \geq 0} P_x \left[{\mathcal A}_1 = S', X_{l+1} = x_i^+, X_l \in B, L_1 \geq l  \right]  P_{x_i^+}[Y_\reg = x_{i+1} , H_B > \reg]\\
&\qquad = P_x [{\mathcal A}_1 =S', X_{L_1+1} = x_i^+, Y_{U_1} = x_{i+1}].
\end{split}
\end{equation}
Summing \eqref{dec0} over $\bar{S}$ and $\bar{\mathbf x}$ and making inductive use of \eqref{dec1}, we infer that
\begin{equation*}
\begin{split}
&Q_0 [{\mathcal A}=S, \mathcal{X} = \mathbf{x}] = \sum_{x_1^+ \in \partial_e B} P[X_{L_1+1} = x_1^+] \prod_{i=2}^k \left( q_{x_{i-1}^+} [Y_\reg = x_i] P_{x_i}[{\mathcal A}_1 = S_i, X_{L_1+1}=x_i^+] \right)\\
&=P[Y_{U_1} = x_2] \left(\prod_{i=2}^{k-1} P_{x_i} [{\mathcal A}_1 = S_i, X_{L_1+1}= x_i^+, Y_{U_i} = x_{i+1} ] \right) P_{x_k} [{\mathcal A}_1 = S_k, X_{L_1+1}=x_k^+]\\
&= P \left[{\mathcal A}=S, (Y_{U_{i-1}}, X_{L_i+1})_{i=2}^k = \mathbf{x} \right] \text{ (by the strong Markov property),}
\end{split}
\end{equation*}
which implies the required identity $Q_0[{\mathcal A}=S] = P[{\mathcal A}=S]$.
Finally, by \eqref{dec0}, $\mathcal A$ and $\bar{\mathcal A}$ are different under $Q_0$ only on the event $\{ \mathcal{X} \neq \bar{\mathcal X}\}$, which by \eqref{dec00} and \eqref{dec0} occurs with probability at most $ke^{-c_\epsilon \log^2 N}$, proving \eqref{e:decouple}. 
\end{proof}

Next, we estimate how many of the excursions defined in \eqref{e:exc} typically occur until time $uN^d$. We set
\begin{align}
\label{e:k}
 k^\pm = \left[ (1 \pm \epsilon) u \, \capacity ({\mathsf A}) \right],
\end{align}
and prove the following estimate:

\begin{lemma}
\label{l:poisson}
\textup{($d \geq 3$)}
\begin{align}
\label{e:poisson1} 
&P [ R_{k^+} \leq uN^d ] \leq
e^{-c_{u, \epsilon} \capacity ({\mathsf A})},   \\
\label{e:poisson2} 
&P [ R_{k^-} \geq uN^d] \leq e^{-c_{u, \epsilon} \capacity ({\mathsf A})}.
\end{align}
\end{lemma}

\begin{proof}
For ease of notation we write
\begin{align*}
 s_N = \inf_{y \in {\mathbb T} \setminus C} E_y [H_A], \text{ and } t_N = \sup_{y \in {\mathbb T}} E_y [H_A]
\end{align*}
throughout this proof. We begin with the proof of \eqref{e:poisson1}.
The observation that $R_k \geq H_A \circ \theta_{U_1} + \cdots + H_A \circ \theta_{U_{k-1}}$, $P$-a.s., the exponential Chebychev inequality and an inductive application of the strong Markov property yield, for any $\nu>0$,
\begin{align}
\label{eld1}
P [ R_{k^+} \leq uN^d] \leq e^{\nu u \frac{N^d}{s_N}} \sup_{y \in {\mathbb T} \setminus B} E_y \left[ e^{- \frac{\nu}{s_N} H_A} \right]^{k^+ -1}.
\end{align}
Next, we bound the expectation with help of the inequality $e^{-t} \leq 1 - t + \frac{t^2}{2}$, valid for all $t \geq 0$, and find
\begin{align*}
 \sup_{y \in {\mathbb T} \setminus B} E_y \left[ e^{-\frac{\nu}{s_N} H_A} \right] \leq 1 - \nu + \frac{\nu^2}{2} \frac{\sup_{y \in {\mathbb T}} E_y [H_A^2]}{s_N^2}.
\end{align*}
In the following estimate, we apply Lemma~\ref{l:kac} to the numerator and \eqref{e:fbd} to the denominator in the first, then \eqref{e:flbd} in the second step,
\begin{align*}
 \frac{\sup_{y \in {\mathbb T}} E_y [H_A^2]}{\inf_{y \in {\mathbb T} \setminus B} E_y [H_A]^2} \leq c_\epsilon \frac{\sup_{y \in {\mathbb T}} E_y [H_A]^2}{E[H_A]^2} \leq c_\epsilon'.
\end{align*}
Hence, we can infer with \eqref{eld1} that
\begin{align}
 \label{eld2}
 P[R_{k^+} \leq u N^d] & \leq \exp \left( \nu u \frac{N^d}{s_N} - \nu (k^+ -1) + c_\epsilon \nu^2 (k^+ -1) \right)\\
&\stackrel{\eqref{e:k}}{\leq} \exp \left( \nu u \frac{N^d}{s_N} - (\nu + c_\epsilon \nu^2) (1 + \epsilon) u \capacity ({\mathsf A}) + c_{\nu, \epsilon} \right). \nonumber
\end{align}
By \eqref{e:fbd} and Proposition~\ref{p:Gloc}, we have $
 \frac{N^d}{s_N} \leq \capacity({\mathsf A}) (1+ \epsilon/2), \text{ for } N \geq c_{\epsilon}.$
The desired estimate \eqref{e:poisson1} follows from \eqref{eld2} by setting $\nu$ equal to a small constant $c_{u,\epsilon}>0$.

In order to prove \eqref{e:poisson2}, we use that, $P$-a.s.,
\begin{equation}
\label{eld3}
\begin{split}
\{R_k \geq uN^d \} \subseteq & \left\{H_A+ H_A \circ \theta_{U_1} + \cdots + H_A \circ \theta_{U_{k-1}} \geq (1- \epsilon/2) u N^d \right\} \\
&\cup \left\{ U \circ \theta_{R_1} + \cdots + U \circ \theta_{R_{k-1}} \geq (\epsilon/2)u N^d \right\}.
\end{split}
\end{equation}
Using again the exponential Chebychev inequality and inductive applications of the strong Markov property, we deduce from \eqref{eld3} that, for any $\theta>0$,
\begin{align}
\label{poisson0}
P [ R_{k^-} \geq uN^d ] &\leq e^{- \theta (1- \epsilon/2) u \frac{N^d}{t_N} } \sup_{x \in {\mathbb T}} E_x \left[ e^{\theta \frac{H_A}{t_N}} \right]^{k^-}+ e^{-(\epsilon/2)u \frac{N^d}{t_N}} \sup_{x \in A} E_x \left[e^{ \frac{U}{t_N}} \right]^{k^-}.  
\end{align}
In order to bound the first expectation on the right-hand side, note that, by Lemma~\ref{l:kac}, we have for $\theta \in (0,\frac{1}{2})$,
\begin{equation}
\label{poisson1}
E \left[ e^{ \theta \frac{H_A}{t_N} } \right] = \sum_{k=0}^\infty \frac{\theta^k}{k \text{!} t_n^k } E[H_A^k] \leq \sum_{k=0}^\infty \theta^k = \frac{1}{1- \theta}.
\end{equation}
In order to deal with the second expectation on the right-hand side of \eqref{poisson0}, we note that, $P_x$-a.s.~for any $x \in A$,
\begin{align*}
U &\leq (\reg + T_C) \mathbf{1} \{H_B \circ \theta_{T_C} > \reg\} + \left( \reg + T_C + U \circ \theta_{H_B} \circ \theta_{T_C} \right) \mathbf{1} \{H_B \circ \theta_{T_C} \leq \reg\}\\
&= \reg + T_C + U \circ \theta_{H_B} \circ \theta_{T_C} \mathbf{1} \{H_B \circ \theta_{T_C} \leq \reg\},
\end{align*}
hence by the strong Markov property,
\begin{equation}
\label{poisson2}
\begin{split}
\sup_{x \in B} E_x [e^{U /t_N}] &\leq \sup_{x \in B} E_x[e^{(\reg+T_{C})/t_N}] \left(1 +  \sup_{y \in {\mathbb T} \setminus C} P_y [H_B \leq \reg] \sup_{x \in B} E_x [e^{U /t_N}] \right)\\
&\leq \sup_{x \in B} E_x[e^{(\reg+T_{C})/t_N}] \left(1 +  \rbdd \sup_{x \in B} E_x [e^{U /t_N}] \right),
\end{split}
\end{equation}
where we have used \eqref{a:Drat} for the second line. By an elementary estimate on simple random walk, we have $cN^2 \leq \sup_{x \in B} E_x[T_{C}] \leq N^2$, hence by Lemma~\ref{l:Hlbd} and \eqref{a:capest},
\begin{align}
 \label{poisson3}
 \frac{1}{t_N} \leq \frac{1}{E[H_A]} \leq c_\epsilon \frac{N^{-\epsilon(d-2)}}{N^2} \leq c_\epsilon \frac{N^{-\epsilon (d-2)}}{\sup_{x \in B} E_x[T_{C}]}.
\end{align}
If we apply Lemma~\ref{l:kac} with $V = {\mathbb T} \setminus C$, we therefore find that $\sup_{x \in B} E_x[e^{T_{C}/t_N}] \leq e^{c_\epsilon N^{-\epsilon}}.$
With this estimate and $\reg/t_N \leq c N^{-\epsilon/2}$ (cf.~\eqref{poisson3}) applied to the right-hand side of \eqref{poisson2}, we obtain
\begin{align}
 \label{poisson4}
\sup_{x \in B} E_x [e^{U /t_N}] \leq e^{c_\epsilon N^{-\epsilon/2}}  \left( 1 +  c_\epsilon \rbdd  \right) \leq e^{c_\epsilon'N^{-\epsilon/2}}.
\end{align}
Substituting \eqref{poisson1} and \eqref{poisson4} into \eqref{poisson0} and using that $(1- \theta)^{-1} \leq 1 + \theta + 2 \theta^2$ for $0 \leq \theta \leq \frac{1}{2}$, we deduce that
\begin{align}
\nonumber
P [ R_{k^-} \geq uN^d ] &\leq \exp \Big(- \theta \Big(1- \frac{\epsilon}{2} \Big) u \frac{N^d}{t_N}  + (\theta + 2 \theta^2) k^- \Big) + \exp \Big( - \frac{\epsilon}{2} u \frac{N^d}{t_N} + c_\epsilon N^{-\epsilon/2} k^- \Big)\\
\label{poisson5} 
&\stackrel{\eqref{e:k}}{\leq} \exp \left(- \theta \Big(1- \frac{\epsilon}{2} \Big) u \frac{N^d}{t_N} + (\theta + 2 \theta^2) ( 1 - \epsilon) u \capacity({\mathsf A}) + c_\theta \right)  \\
\nonumber
&\qquad + \exp \left( - \frac{\epsilon}{2} u \frac{N^d}{t_N} + c_\epsilon N^{-\epsilon/2} (1 - \epsilon ) u \capacity({\mathsf A}) + c_\epsilon \right).
\end{align}
Again, we apply \eqref{e:fbd} and Proposition~\ref{p:Gloc} and find that for $N \geq c_{\epsilon}$,
$\frac{N^d}{t_N} \geq \capacity({\mathsf A}) \Big( 1 - \frac{\epsilon}{2} \Big),$
so that \eqref{e:poisson2} follows from \eqref{poisson5} upon choosing $\theta$ as a small constant $c_{u, \epsilon}>0$.
\end{proof}

We now introduce
\begin{equation}
\label{d:nnt}
\begin{split}
&\text{the space $\Gamma$ of cadlag functions $w$ from $[0,\infty)$ to $\mathbb T$ with at most finitely many} \\
&\text{discontinuities and such that $w_0 \in \partial_i A$,}
\end{split}
\end{equation}
endowed with the canonical $\sigma$-algebra ${\mathcal F}_\Gamma$ generated by the coordinate projections, as well as
\begin{align}
\label{d:meas}
\text{the space $M(\Gamma)$ of finite point measures on $\Gamma$,}
\end{align}
endowed with the $\sigma$-algebra ${\mathcal F}_{M(\Gamma)}$ generated by the evaluation maps $e_A: \mu \mapsto \mu(A)$, $A \in {\mathcal F}_\Gamma$. On the space $({\bar \Omega}, {\bar {\mathcal F}}, {\bar P}_\sigma)$ (cf.~\eqref{e:Yhat}, \eqref{e:J}), we define $\mu^\pm_1$ by
\begin{align}
\label{d:mu1}
\mu_1^\pm = \sum_{2 \leq i \leq 1+J^\pm} \delta_{{\bar Y}^{i}} \in M(\Gamma),
\end{align}
where $\delta_w$ denotes the Dirac mass at $w \in \Gamma$. We then define the random sets
\begin{align}
\label{d:I1}
{\mathcal I}_1^\pm = \bigcup_{w \in \supp (\mu_1^\pm)} \ran (w) \subseteq {\mathbb T}.
\end{align}
Note that by \eqref{e:Yhat} and \eqref{e:J},
\begin{equation}
\label{e:mu1}
\begin{split}
&\text{the random measures $\mu_1^\pm$ are Poisson point measures on $\Gamma$ with intensity}\\
&\text{measures $(1 \pm 2 \epsilon)u \capacity( {\mathsf A})  \kappa_1$, where $\kappa_1$ is the law of $(Y_{t \wedge U_1})_{t \geq 0}$ under $P_\sigma$.}
\end{split}
\end{equation}
The following proposition contains a first coupling of the trajectory $Y_{[R_2, uN^d]}$ with random point measures. Note that we do not consider the trajectory before time $R_2$. The reason is that Lemma~\ref{l:decouple} does not provide an estimate on the distribution of the first entrance point $Y_{R_1}$. This problem will be dealt with separately in Lemma~\ref{l:f} below.

\begin{proposition} \label{p:c1}
\textup{($d \geq 3$)}
There is a coupling $(\Omega_1, {\mathcal F}_1, Q_1)$ of $Y_{[R_2,uN^d]}$ under $P$ with $\mu_1^\pm$ under ${\bar P}_\sigma$, such that
\begin{align}
\label{e:c1}
&Q_1 \left[ {\mathcal I}_1^- \cap A \subseteq Y_{[R_2,uN^d]} \cap A \subseteq  {\mathcal I}_1^+ \cap A\right]   \geq 1 - e^{-c_{u, \epsilon} \log^2 n}.
\end{align}
\end{proposition}

\begin{proof}
Denoting the total number of excursions started before time $uN^d$ by $K_u = \sup \{k \geq 0: R_k \leq u N^d \}$, we have
\begin{align}
\label{q10}
\cup_{i=2}^{K_u -1} Y_{[R_i,U_i]} \cap A \subseteq Y_{[R_2,uN^d]} \cap A \subseteq \cup_{i=2}^{K_u} Y_{[R_i,U_i]} \cap A.
\end{align}
By Lemma~\ref{l:decouple}, we can couple $(Y_{[R_i,U_i]} \cap A)_{i=2}^k$ under $P$ with $({\bar Y}^i_{[R_1,U_1]} \cap A)_{i=2}^k$ under ${\bar P}_\sigma$, such that these two random vectors differ with probability at most $ke^{-c_\epsilon \log^2 n}$, where we choose 
\begin{align}
\label{q11}
k = [2  u \capacity({\mathsf A}) ] \leq c_{\epsilon, u} N^{d-2}, \text{ cf.~\eqref{a:capest}.}
\end{align}
Given  $(Y_{[R_i,U_i]} \cap A)_{i=2}^k$ and $({\bar Y}^i_{[R_1,U_1]} \cap A)_{i=2}^k$, we extend this coupling with two conditionally independent random vectors $(Y_{[R_i,U_i]} \cap A)_{i=k+1}^\infty \in (2^A)^{\mathbb N}$  and $({\bar Y}^i)_{i \geq 2} \in \Gamma^{\mathbb N}$, distributed as $(Y_{[R_i,U_i]} \cap A)_{i=k+1}^\infty$ given $(Y_{[R_i,U_i]} \cap A)_{i=2}^k$ under $P$ and as $({\bar Y}^i)_{i \geq 2}$ given $({\bar Y}^i_{[R_1,U_1]} \cap A)_{i=2}^k$ under ${\bar P}_\sigma$.
Adding independent Poisson variables $J^-$ and $J^+$ as in \eqref{e:J}, we thus obtain a coupling $q$ of $(Y_{[R_i,U_i]} \cap A)_{i \geq 2}$ under $P$, $({\bar Y}^i)_{i \geq 2}$, $J^-$ and $J^+$ under ${\bar P}_\sigma$, such that
\begin{equation}
\label{q12}
\begin{split}
&q\left[ \begin{array}{c} (Y_{[R_i,U_i]} \cap A)_{i=2}^k = ({\bar Y}_{[R_1,U_1]}^i \cap A)_{i=2}^k, \\ J^- \leq k^- \leq k^+ \leq J^+ \end{array} \right] \geq 1 - e^{-c_{u, \epsilon} \log^2 N},
\end{split}
\end{equation}
where we have also used Lemma~\ref{l:poissonld} with the definition of $k^\pm$ in \eqref{e:k}. Note that $\mu_1^\pm$ and ${{\mathcal I}_1}^\pm$ can be defined under $q$ as in \eqref{d:mu1} and \eqref{d:I1} and by construction of $({\bar Y}^i)_{i \geq 2}$,  \eqref{e:mu1} applies. We now define the coupling $Q_1$ by specifying the distribution of $(Y_{[R_2,uN^d]},\mu_1^-, \mu_1^+)$ on $2^{\mathbb T} \times M(\Gamma)^2$. For any $R \subseteq {\mathbb T}$ and $M_1, M_2 \in {\mathcal F}_{M({\Gamma})}$, we set
\begin{equation}
\begin{split}
&Q_1 \left[ Y_{[R_2,uN^d]} = R, \mu_1^- \in M_1, \mu_1^+ \in M_2 \right] =\\
&\qquad \sum_{S \subseteq A} P \left[ Y_{[R_2,uN^d]} = R, \cup_{i=2}^k Y_{[R_i,U_i]} \cap A = S \right]\\
&\qquad \qquad \times q \left[\mu_1^- \in M_1, \mu_1^+ \in M_2 \Big| \cup_{i=2}^k Y_{[R_i,U_i]} \cap A = S \right],
\end{split}
\end{equation}
where the term in the sum is understood to equal $0$ if $P[\cup_{i=2}^k Y_{[R_i,U_i]} \cap A = S]=0$.
Then we have $Q_1[Y_{[R_2,uN^d]} =R] = P [Y_{[R_2,uN^d]} =R]$, as well as by \eqref{e:mu1}, $Q_1 [\mu_1^- \in M_1 ] = {\bar P}_\sigma [\mu_1^- \in M_1 ]$ and $Q_1 [\mu_1^+ \in M_2 ] = {\bar P}_\sigma [\mu_1^+ \in M_2 ]$ for any $R \subseteq {\mathbb T}$, $M_1, M_2 \in {\mathcal F}_{M({\Gamma})}$, so $Y_{[R_2,uN^d]}$, $\mu_1^-$ and $\mu_1^+$ have the correct distributions under $Q_1$. Moreover, we have by \eqref{d:mu1} and \eqref{q10},
\begin{align*}
&Q_1 \left[ \left\{{\mathcal I}_1^- \cap A \subseteq Y_{[R_2,uN^d]} \cap A \subseteq  {\mathcal I}_1^+ \cap A \right\}^c\right] \leq  q \left[ (Y_{[R_i,U_i]} \cap A)_{i=2}^k \neq ({\bar Y}^i_{[R_1,U_1]} \cap A)_{i=2}^k \right] \\
& + q \left[ k^- \leq J^- \right] + q \left[ J^+ \leq k^+ \right] + q \left[ k<J^+ \right] + P \left[ \left\{ k^-  \leq K_u - 1 \leq K_u \leq k^+ \right\}^c \right],
\end{align*}
Using \eqref{q12}, Lemma~\ref{l:poisson} together with \eqref{a:capest} and a large deviations bound on $q [k< J^+ ]$ similar to Lemma~\ref{l:poissonld}, we find that the right-hand side is bounded by $e^{-c_{u, \epsilon} \log^2 N}$, as required.
\end{proof}

The final step in this section is to modify the above coupling in such a way that the random paths in the Poisson clouds have starting points distributed according to the normalized equilibrium measure of $A$ (cf.~\eqref{e:eqfin}), as do random interlacement paths (cf.~\eqref{e:interlac}). For this purpose, we define the measure
\begin{align}
\label{d:k2}
\kappa_2  \text{ as the law on $(\Gamma,{\mathcal F}_\Gamma)$ of $(Y_{t \wedge U_1})_{t \geq 0}$ under } P_{e_A}
\end{align} 
(note that $\kappa_2(\Gamma)= \capacity ({\mathsf A})$), and in the following lemma relate $\kappa_2$ to the intensity measures of $\mu_1^\pm$ (cf.~\eqref{e:mu1}).

\begin{lemma} \label{l:intest}
For $N \geq c_{u, \epsilon}$, 
\begin{align}
\label{e:intest}
(1-3 \epsilon) u \kappa_2 \leq (1-2\epsilon)  u \capacity({\mathsf A}) \kappa_1  \leq (1+2\epsilon) u \capacity({\mathsf A}) \kappa_1 \leq (1+3 \epsilon) u \kappa_2.
\end{align}
\end{lemma}

\begin{proof}
Since $\capacity({\mathsf A}) \kappa_1 = \capacity({\mathsf A})P_\sigma[Y_{H_A}=w_0] e_A(w_0)^{-1} \kappa_2,$ the statement follows from Lemma~\ref{l:quni}.
\end{proof}

The last lemma now allows us to construct the required coupling.

\begin{proposition} \label{p:c2}
\textup{($d \geq 3$)}
There is a coupling $(\Omega_2, {\mathcal F}_2, Q_2)$ of $Y_{[R_2,uN^d]}$ under $P$ with Poisson random point measures $\mu_2^\pm$ on ${\Gamma}$ (cf.~\eqref{d:nnt}) with intensity measures $(1\pm 3\epsilon)u \kappa_2$ (cf.~\eqref{d:k2}), such that
\begin{align}
\label{e:c2}
&Q_2 \left[ {\mathcal I}_2^- \cap A \subseteq Y_{[R_2,uN^d]} \cap A \subseteq  {\mathcal I}_2^+ \cap A \right]   \geq 1 -  e^{-c_{u, \epsilon} \log^2 N}, \text{ where}
\end{align}
\begin{align}
\label{d:I2}
{\mathcal I}_2^\pm = \bigcup_{w \in \supp \mu_2^\pm} \ran (w).
\end{align}
\end{proposition}

\begin{proof}
Note that for $N \geq c_{u, \epsilon}$, the inequalities in Lemma~\ref{l:intest} hold. For such $N$, we can therefore construct independent Poisson random measures $\nu_1, \nu_2, \nu_3$ and $\nu_4$ on $\Gamma$ with intensity measures $(1-3\epsilon) u \kappa_2$, $(1-2\epsilon) u \capacity ({\mathsf A})  \kappa_1 - (1-3 \epsilon) u \kappa_2 \geq 0$, $4 \epsilon u  \capacity ({\mathsf A}) \kappa_1$ and $(1+3 \epsilon) u \kappa_2 - (1+2 \epsilon) u \capacity ({\mathsf A}) \kappa_1 \geq 0$. Then $\nu_1 \leq \nu_1 + \nu_2 \leq \nu_1 + \nu_2 + \nu_3 \leq \nu_1 + \nu_2 + \nu_3 + \nu_4$ are random measures with the distributions of $\mu_2^-$, $\mu_1^-$, $\mu_1^+$ and $\mu_2^+$ (cf.~\eqref{e:mu1}). We have thus constructed a coupling $q$ of $\mu_2^\pm$ and $\mu_1^\pm$, such that (see \eqref{d:I1} and \eqref{d:I2})
\begin{align}
\label{c21}
{{\mathcal I}_2}^- \subseteq {\mathcal I}_1^- \subseteq {\mathcal I}_1^+ \subseteq {{\mathcal I}_2}^+, \text{ $q$-a.s.}
\end{align}
Together with the coupling $Q_1$ from Proposition~\ref{p:c1}, we now define the coupling $Q_2$ as follows: For any $S \subseteq {\mathbb T}$, $M_1, M_2 \in {\mathcal F}_{M({\Gamma})}$, we set
\begin{equation*}
\begin{split}
&Q_2 \left[ Y_{[R_2,uN^d]} = S, \mu_2^- \in M_1,  \mu_2^+ \in M_2 \right]= \\
&\sum_{S_1, S_2 \subseteq {\mathbb T}} Q_1 \left[ Y_{[R_2,uN^d]}  = S, {\mathcal I}_1^- =S_1, {\mathcal I}_1^+ = S_2 \right] q \left[ \mu_2^- \in M_1, \mu_2^+ \in M_2 \Big| {\mathcal I}_1^- =S_1, {\mathcal I}_1^+ = S_2 \right],
\end{split}
\end{equation*}
where the term in the sum equals $0$ by convention whenever $Q_1[{\mathcal I}_1^- = S_1, {\mathcal I}_1^+ = S_2]=0$. Then Proposition~\ref{p:c1} and the construction of $q$ imply that $Y_{[R_2,uN^d]}$, $\mu_2^-$ and $\mu_2^+$ have the correct distributions under $Q_2$. Finally, \eqref{c21} and \eqref{e:c1} together yield \eqref{e:c2}. 
\end{proof}

%% file: above.tex
\section{Domination by random interlacements}
\label{s:dom+}

The purpose of this section is to prove one half of Theorem~\ref{t:dom} in Proposition~\ref{p:dom+}. This proposition shows that the random walk trajectory on $\mathbb T$ can be coupled with the trace of a random interlacement on $\mathsf A$ such that the image of the random walk trajectory in $\mathsf A$ is a subset of the random interlacement with high probability. The main work appears in Proposition~\ref{p:pdom+}, where we decompose the random set ${\mathcal I}^+_2 \cap A$ appearing in Proposition~\ref{p:c2} into two independent sets, one of which is empty with high probability, the other one of which is stochastically dominated by a random interlacement intersected with $\mathsf A$. The proof involves truncation of the trajectories of ${\mathcal I}^+_2$. A small increase of the intensity parameter from $u(1+3 \epsilon)$ to $u(1+4 \epsilon)$ in the dominating random interlacement compensates for the truncation. The arguments follow the ones of Sznitman in \cite{Szn09c}, where a similar procedure is carried out for random walk trajectories on discrete cylinders.

\medskip

In order to state the first proposition, we construct on some auxiliary probability space $(\Omega',{\mathcal F}',Q')$ for any $u>0$ and $\epsilon \in (0,1/4)$,
\begin{equation}
\label{e:tuncYi}
\begin{split}
&\text{an iid sequence $Y^i$, $i \geq 1$, of random walks with same distribution as}\\
&\text{$(Y_{t \wedge T_{C}})_{t \geq 0}$ under $P_{e_A}/\capacity({\mathsf A})$,}
\end{split}
\end{equation}  
\begin{equation}
\text{an independent Poisson variable $J$ with parameter $(1 + 4 \epsilon) u \capacity({\mathsf A})$.}
\end{equation}
This enables to define the Poisson point measure on $\Gamma$ (cf.~\eqref{d:nnt}):
\begin{equation}
\label{e:mu}
\mu = \sum_{1 \leq i \leq J} \delta_{Y^i} \in M(\Gamma).
\end{equation}
Then for $N \geq c_\epsilon$, 
\begin{equation}
\label{e:muint}
\begin{split}
&\text{$\mu$ is a Poisson point measure with intensity measure $(1 + 4 \epsilon) u \kappa$ on $\Gamma$,}\\
&\text{where $\kappa$ is the law of $(Y_{t \wedge T_{C}})_{t \geq 0}$ under $P_{e_A}$.}
\end{split}
\end{equation}
We then define
\begin{align}
{\mathcal I} = \bigcup_{w \in \supp \mu} \ran (w),
\end{align}
so that if the paths were not cut off when leaving $\mathsf C$, then $\phi({\mathcal I} \cap A)$ would have the distribution of a random interlacement ${\mathcal I}^{u(1+4 \epsilon)}$ intersected with $\mathsf A$ (see \eqref{e:interlac}). In particular, by (1.20), (1.43) and (1.45) in \cite{Szn09}, 
\begin{equation}
 \label{e:dom00}
\begin{split} 
&\phi( {\mathcal I} \cap A) \text{ is stochastically dominated by } {\mathcal I}^{u(1+4 \epsilon)} \cap {\mathsf A} \text{ under } {\mathbb P}.
\end{split}
\end{equation}

\begin{proposition} \label{p:pdom+}
\textup{($d \geq 3$)}
For any $\alpha>0$, there exist random subsets ${\mathcal I}^*$ and $\bar {\mathcal I}$ of $A$, defined on $(\Omega_2, {\mathcal F}_2, Q_2)$ of Proposition~\ref{p:c2}, such that for $N \geq c_{\epsilon, \alpha}$,
\begin{align}
\label{e:dom+1}&{\mathcal I}_2^+ \cap A = {\mathcal I}^* \cup {\bar {\mathcal I}},\\
\label{e:dom+2}&{\mathcal I}^* \text{ and } {\bar {\mathcal I}} \text{ are independent under } Q_2,\\
\label{e:dom+3}&Q_2[{\bar {\mathcal I}} \neq \emptyset] \leq c_{\epsilon, \alpha} u N^{-\alpha},\\
\label{e:dom+4}&{\mathcal I}^* \text{ is stochastically dominated by } {\mathcal I} \cap A.
\end{align}
\end{proposition}

\begin{proof}
The decomposition \eqref{e:dom+1} will depend on the number of excursions between $A$ and the complement of the ball
\begin{align*}
 B' = B(0,N^{1-\epsilon/2}/2) \subset B, \text{ cf.~\eqref{e:boxes},}
\end{align*}
made by the random paths. Hence, we define on $\Gamma$ the return and departure times \begin{equation}
\begin{split}
&{\tilde R}_1 = H_A, \, {\tilde D}_1 = {\tilde R}_1 + T_{B'} \circ \theta_{{\tilde R}_1}, \text{ and for } l \geq 2,\\
&{\tilde R}_l = {\tilde D}_{l-1} + {\tilde R}_1 \circ \theta_{{\tilde D}_{l-1}}, \, {\tilde D}_l = {\tilde D}_{l-1} + {\tilde D}_1 \circ \theta_{{\tilde D}_{l-1}},
\end{split}
\end{equation}
where by convention, $\inf \emptyset = \infty$ in the definition of $H_A$ and $T_{B'}$, cf.~\eqref{e:H}, \eqref{e:T}. By \eqref{a:gotoC}, \eqref{a:Drat'}, and the Markov property applied at time $T_C$, we have for $U_1$ as in \eqref{e:exc},
\begin{align}
\label{e:domu-2}
\sup_{x \in \partial_e B'} P_x [ H_A < U_1 ] \leq 2 N^{-c_{1,\epsilon}},
\end{align}
for some constant $c_{1,\epsilon}>0$. We fix 
\begin{align}
\label{e:domu-1}
m =  \left[ (\alpha+d)/c_{1,\epsilon} \right] + 1,
\end{align}
and introduce the decomposition
\begin{equation}
\label{e:domu-0.1}
\begin{split}
\mu = \sum_{l \geq 1} \mu_l, \text{ where } \mu_l = \mathbf{1} \{ {\tilde D}_l < T_{C} < {\tilde R}_{l+1} \} \mu, \text{ for } l \geq 1.
\end{split}
\end{equation}
as well as
\begin{equation}
\label{e:domu0}
\begin{split}
\mu_2^+ = \sum_{1 \leq l \leq m} \mu_2^{+,l} + {\bar \mu}, \text{
where } \mu_2^{+,l} = \mathbf{1} \{{\tilde D}_l < U_1 < {\tilde
R}_{l+1} \} \mu_2^+,
\end{split}
\end{equation}
for $l \geq 1$, and $\bar \mu = \mathbf{1} \{{\tilde D}_{m+1} < U_1 \} \mu_2^+$.
Observe that
\begin{equation}
\label{e:domu0.1}
\mu_2^{+,l}, 1 \leq l \leq m, \text{ and } {\bar \mu} \text{ are independent Poisson measures under } Q_2, \text{ and}
\end{equation}
\begin{equation}
\label{e:domu0.2}
\mu_l, l \geq 1, \text{ are independent Poisson measures under } Q'.
\end{equation}
(recall the definition of $Q'$ above \eqref{e:tuncYi}). We define
\begin{equation}
\label{e:domu1}
{\mathcal I}^* = \bigcup_{1 \leq l \leq m} \left( \bigcup_{w \in \supp \mu_2^{+,l}} \ran (w) \cap A \right), \, {\bar {\mathcal I}} = \bigcup_{w \in \supp {\bar \mu}} \ran (w) \cap A,
\end{equation}
so that by \eqref{d:I2}, \eqref{e:domu0} and \eqref{e:domu0.1},
\begin{equation}
\label{e:indep}
{\mathcal I}^+_2 \cap A = {\mathcal I}^* \cup {\bar {\mathcal I}}, \text{ and } {\mathcal I}^*, {\bar {\mathcal I}} \text{ are independent under } Q_2.
\end{equation}
Moreover, one has by \eqref{e:domu-0.1}, 
\begin{equation}
\label{e:domu2}
{\mathcal I} \cap A = \bigcup_{l \geq 1} \left( \bigcup_{w \in \supp \mu_l} \ran (w) \cap A \right).
\end{equation}
For $l \geq 1$, we introduce the map $\phi_l'$ from $\{{\tilde D}_l < U_1 < {\tilde R}_{l+1} \} \subseteq \Gamma$ into $W_f^{\times l}$, where $W_f$ denotes the countable collection of finite nearest neighbor paths with values in $B' \cup \partial_e B'$, as well as the map $\phi_l$ from $\{{\tilde D}_l < T_{C} < {\tilde R}_{l+1}\} \subseteq \Gamma$ into $W^{\times l}_f$ defined by
\begin{equation}
\label{e:domu3}
\begin{split}
&\phi'_l(w) = \Bigl( \left(  w_{\tau_{n + N_{{\tilde R}_k}}} : 0 \leq n \leq N_{{\tilde D}_k}-N_{{\tilde R}_k}  \right) \Bigr)_{1 \leq k \leq l}, \text{ for } w \in \{ {\tilde D}_l < U_1 < {\tilde R}_{l+1} \}, \\
&\phi_l(w) = \Bigl( \left( w_{\tau_{n+N_{{\tilde R}_k}}} : 0 \leq n \leq N_{{\tilde D}_k}-N_{{\tilde R}_k}  \right) \Bigr)_{1 \leq k \leq l}, \text{ for } w \in \{ {\tilde D}_l < T_{C} < {\tilde R}_{l+1} \}.
\end{split}
\end{equation}
Intuitively speaking, the maps $\phi_l$ and $\phi'_l$ chop the trajectories into their successive excursions between $A$ and $(B')^c$. We can respectively view $\mu_2^{+, l}$ and $\mu_{l}$ for $l \geq 1$ as Poisson point processes on $\{{\tilde D}_l < U_1 < {\tilde R}_{l+1}\}$ and $\{{\tilde D}_l < T_{C} < {\tilde R}_{l+1}\}$. If $\rho_{+,l}$ and $\rho_{l}$ denote their respective images under $\phi'_l$ and $\phi_l$, we see from \eqref{e:domu0.1} and \eqref{e:domu0.2} that
\begin{equation}
\label{e:domu4}
\begin{split}
&\rho_{+, l}, 1 \leq l \leq m, \text{ and } {\bar \mu} \text{ are independent Poisson point processes, and}\\
&\rho_l, 1 \leq l, \text{ are independent Poisson point processes,}
\end{split}
\end{equation}
and denoting by $\xi_{+,l}$ and $\xi_l$ the intensity measures on $W^{\times l}_f$ of $\rho_{+,l}$ and $\rho_l$, we have:
\begin{equation}
\label{e:domu5}
\begin{split}
& \xi_{+,l} (dw_1, \ldots, dw_l) = (1 + 3 \epsilon) P_{e_A} \left[ \begin{array}{c} {\tilde D}_l < U_1 < {\tilde R}_{l+1}, \\
 (X_{n+N_{{\tilde R}_k}})_{0 \leq n \leq N_{{\tilde D}_k} - N_{{\tilde R}_k}} \negmedspace \in dw_k, 1 \leq k \leq l \end{array} \right],\\
& \xi_l (dw_1, \ldots, dw_l) = (1 + 4 \epsilon) P_{e_A} \left[ \begin{array}{c} {\tilde D}_l < T_{C} < {\tilde R}_{l+1}, \\
(X_{n+N_{{\tilde R}_k}})_{0 \leq n \leq N_{{\tilde D}_k} - N_{{\tilde R}_k}} \negmedspace \in dw_k,  1 \leq k \leq l \end{array} \right].
\end{split}
\end{equation}

\begin{lemma} \label{l:xidom}
For $N \geq c_{\epsilon, \alpha}$,
\begin{align}
\label{e:domu6}
\xi_{+,l} \leq \xi_l, \text{ for } 1 \leq l \leq m.
\end{align}
\end{lemma}

\begin{proof}[Proof of Lemma~\ref{l:xidom}.]
Let $x \in \partial_e B'$. By applying the strong Markov property at the times $T_{C} \leq H_{B'} \circ \theta_{T_{C}} + T_{C}$, we obtain
\begin{align*}
 P_x \left[ T_{C} < H_A < U_1, Y_{H_A} = y \right] &\leq \sup_{x' \in {\mathbb T} \setminus {C}} P_{x'} [ H_B \leq \reg ] \sup_{x'' \in \partial_e B'} P_{x''} \left[H_A < U_1, Y_{H_A} = y \right] \\
&\stackrel{\eqref{a:Drat'}}{\leq}  \rbdd \sup_{x' \in \partial_e B'} P_{x'} \left[ H_A < U_1 , Y_{H_A} = y \right].
\end{align*}
By the Markov property applied at time $\tau_1$, the mapping
$z \mapsto P_z \left[ H_A < U_1 , Y_{H_A} = y\right]$
is harmonic on the set $B \setminus A$. Applying the Harnack inequality (cf.~\cite{L91}, Theorem~1.7.2, p.~42) and a standard covering argument, we deduce from the above that, for any $x \in \partial_e B'$,
\begin{align*}
&P_x \left[ T_{C} < H_A < U_1, Y_{H_A} = y \right] \leq c_\epsilon' \rbdd \inf_{x' \in \partial_e B'} P_{x'} \left[ H_A \leq U_1, Y_{H_A}=y \right] \\
&\qquad \leq c_\epsilon' \rbdd \left( P_x \left[ T_{C} \leq H_A \leq U_1, Y_{H_A}=y \right] + P_x \left[ H_A \leq T_{C}, Y_{H_A}=y \right] \right).
\end{align*}
We have hence shown that, for $x \in \partial_e B'$,
\begin{equation} 
\label{e:bend}
 P_x \left[ T_{C} < H_A < U_1, Y_{H_A} = y \right] \leq c_\epsilon' \rbdd P_x \left[ H_A < T_{C}, Y_{H_A}=y \right].
\end{equation}
In order to prove \eqref{e:domu6}, it is sufficient to prove that for $N \geq c_\epsilon$ and $m$ as in \eqref{e:domu-1},
\begin{align}
\label{e:domu6.1}
\xi_{+,l} \leq \frac{1+ 3 \epsilon}{1+ 4\epsilon} \left( 1 + c_\epsilon' \rbdd \right)^{l-1} \xi_l, \text{ for } 1 \leq l \leq m.
\end{align}
Given $w \in W_f$, we write $w^s$ and $w^l$ for the respective starting point and endpoint of $w$. When $w_1, \ldots, w_l \in W_f$ we have 
\begin{equation}
\label{e:domu7}
\begin{split}
&\xi_{+,l} \left( (w_1, \ldots, w_l) \right) \stackrel{\eqref{e:domu5}}{=} \\
& (1+ 3 \epsilon) P_{e_A} [ {\tilde D}_l < U_1 < {\tilde R}_{l+1}, (X_{\cdot + N_{{\tilde R}_k}})_{0 \leq \cdot \leq N_{{\tilde D}_k} - N_{{\tilde R}_k}} = w_k(.), 1 \leq k \leq l ] = \\
&\sum_{I \subseteq \{1, \ldots, l-1\}} (1+ 3 \epsilon) P_{e_A} [ {\tilde D}_l < U_1 < {\tilde R}_{l+1}, (X_{\cdot + N_{{\tilde R}_k}})_{0 \leq \cdot \leq N_{{\tilde D}_k} - N_{{\tilde R}_k}} = w_k(.),  \\
&\quad  1 \leq k \leq l, \text{and } T_{C} \circ \theta _{{\tilde D}_k} + {\tilde D}_k < {\tilde R}_{k+1}, \text{ exactly for } k \in I \text{ when } 1 \leq k \leq l-1 ].
\end{split}
\end{equation}
The above expression vanishes unless $w_k^s \in \partial_i A$ and $w_k^e \in \partial_e B'$ and $w_k$ takes values in $B'$ except for the final point $w_k^e$, for $1 \leq k \leq l$. If these conditions are satisfied, applying the strong Markov property repeatedly at times ${\tilde D}_l, {\tilde R}_l, {\tilde D}_{l-1}, {\tilde R}_{l-1}, \ldots, {\tilde D}_1$, we find that the last member of \eqref{e:domu7} equals 
\begin{align*}
&\sum_{I \subseteq \{1, \ldots, l-1\}} (1+ 3 \epsilon) P_{e_A} [(X_.)_{0 \leq \cdot \leq N_{{\tilde D}_1}} = w_1(.)] E_{w^e_1} [ \mathbf{1} \{1 \notin I\} \mathbf{1} \{H_A < T_{C}\} + \\
&\mathbf{1}\{1 \in I\} \mathbf{1}\{T_{C} < H_A\}, H_A < U_1, Y_{H_A} = w_2^s ] P_{w^s_2} [ (X_.)_{0 \leq \cdot \leq N_{{\tilde D}_1}} = w_2(.)] \ldots \\
&E_{w^s_{l-1}} \left[ \mathbf{1} \{l-1 \notin I\} \mathbf{1} \{H_A < T_{C} \} + \mathbf{1} \{l-1 \in I\} \mathbf{1} \{T_{C} < H_A\}, H_A < U_1, Y_{H_A} = w^s_l \right] \\
&P_{w^s_l} [(X_.)_{0 \leq \cdot \leq N_{{\tilde D}_1}} = w_l( \cdot)] P_{w^e_l} [U_1 < H_A] \\
&\stackrel{\eqref{e:bend}, \, T_C \leq U_1}{\leq} \sum_{I \subseteq \{1, \ldots, l-1\}} (c_\epsilon' \rbdd)^{|I|} (1+3 \epsilon) P_{e_A} [(X_.)_{0 \leq \cdot \leq N_{{\tilde D}_1}} = w_1 (\cdot)] \\
&\qquad P_{w^e_1}[ H_A < T_{C}, Y_{H_A} = w^s_2] P_{w^s_2} [ (X_.)_{0 \leq \cdot \leq N_{{\tilde D}_1}} = w_2(\cdot) ] \ldots \\
&\qquad P_{w^e_{l-1}} [H_A < T_{C}, Y_{H_A} = w^s_l] P_{w^s_l} [ (X_.)_{0 \leq \cdot \leq N_{{\tilde D}_1}} = w_l(\cdot)] P_{w^e_l} [T_C< H_A],
\end{align*}
and using the binomial formula and the strong Markov property, this equals
\begin{align*}
&(1+3 \epsilon) \left( 1 + c_\epsilon' \rbdd \right)^{l-1} P_{e_A} [T_{C} \circ \theta_{{\tilde D}_k} + {\tilde D}_k > {\tilde R}_{k+1}, \text{ for } 1 \leq k \leq l-1,\\
&\qquad \qquad \qquad \qquad (X_{\cdot+ N_{{\tilde R}_k}})_{0 \leq \cdot \leq N_{{\tilde D}_k} - N_{{\tilde R}_k}}  = w_k(\cdot), \text{ for } 1 \leq k \leq l, {\tilde D}_l < T_{C} < {\tilde R}_{l+1} ] \leq \\
&(1 + 3 \epsilon) \left( 1 + c_\epsilon' \rbdd \right)^{l-1} P_{e_A} \left[ \begin{array}{c} {\tilde D}_l < T_{C} < {\tilde R}_{l+1},  (X_{\cdot+ N_{{\tilde R}_k}})_{0 \leq \cdot \leq N_{{\tilde D}_k} - N_{{\tilde R}_k}} = w_k(\cdot), \\ \text{ for } 1 \leq k \leq l \end{array} \right] \\
&= \frac{1+3\epsilon}{1+4 \epsilon} \left(1 + c_\epsilon' \rbdd \right)^{l-1} \xi_l \left( (w_1, \ldots, w_l) \right),
\end{align*}
proving \eqref{e:domu6.1}, as required.
\end{proof}

We now complete the proof of Proposition~\ref{p:pdom+}. By \eqref{e:domu1}, \eqref{e:domu2} and \eqref{e:domu3},
\begin{equation*}
{\mathcal I}^* = \bigcup_{1 \leq l \leq m} \bigcup_{(w_1, \ldots, w_l) \in \supp \rho_{+,l}} (\ran (w_1) \cup \ldots \cup \ran (w_l) ) \cap A, \text{ and}
\end{equation*}
\begin{equation*}
{\mathcal I} \cap A \supseteq \bigcup_{1 \leq l \leq m} \bigcup_{(w_1, \ldots, w_l) \in \supp \rho_l} (\ran (w_1) \cup \ldots \cup \ran (w_l) ) \cap A.
\end{equation*}
Hence, by \eqref{e:domu3} and \eqref{e:domu6}, for $N \geq c_{\epsilon, \alpha}$,
\begin{equation}
\label{e:domu8}
\text{${\mathcal I} \cap A$ under $Q'$ stochastically dominates ${\mathcal I}^*$ under $Q_2$.}
\end{equation}
Finally, by \eqref{e:domu1} and an application of the strong Markov property at the times ${\tilde D}_m,$ ${\tilde D}_{m-1}, \ldots, {\tilde D}_1$,
\begin{equation}
\label{e:domu9}
\begin{split}
Q_2[ {\bar {\mathcal I}} \neq \emptyset] &= Q_2 [ {\bar \mu} \neq 0 ] \stackrel{\eqref{e:domu0}}{=} (1+3 \epsilon)u P_{e_A} [{\tilde R}_{m+1} < U_1 ]\\
&\leq c_\epsilon u \capacity({\mathsf A}) \sup_{x \in \partial_e B'} P_x[H_A < U_1]^m \stackrel{\eqref{e:domu-2}, \eqref{e:domu-1}}{\leq} c_{\epsilon, \alpha} u N^{-\alpha}. 
\end{split}
\end{equation} 
The statements \eqref{e:indep}, \eqref{e:domu8} and \eqref{e:domu9} now complete the proof of Proposition~\ref{p:pdom+}.
\end{proof}

The following lemma will allow us to disregard the first excursion between $R_1$ and $D_1$, when constructing the required coupling (recall the paragraph before Proposition~\ref{p:c1}).

\begin{lemma} 
\label{l:f}
\textup{($d \geq 3$)}
For any $0<u_1<u_2$, there exists a coupling $\tilde q$ of a discrete-time random walk $(X_n)_{n \geq 0}$ on $\mathbb T$ under $P$ and a continuous-time random walk $(Y'_t)_{t \geq 0}$ on $\mathbb T$ under $P$, such that 
\begin{align}
{\tilde q} \left[ X_{[0,u_1 N^d]} \cap A \subseteq Y'_{[R_2,u_2 N^d]} \right] \geq 1-e^{-c_{u_1,u_2} \capacity({\mathsf A})},
\end{align}
where $R_2$ is defined as in \eqref{e:exc} with $Y'$ in place of $Y$.
\end{lemma}

\begin{proof}
All we need to do in the construction of $\tilde q$ is to introduce a uniformly distributed vertex $Y_0=Y'_{\eta N^d} \in {\mathbb T}$, where $\eta = (u_2-u_1)/2$, as well as two independent random walks $(Y_t = Y'_{\eta N^d +t})_{t \geq 0}$ and $(Y'_{\eta N^d -t})_{0 \leq t \leq \eta N^d}$ starting at $Y_0 = Y'_{\eta N^d}$. By reversibility, both $(Y_t)_{t \geq 0}$ and $(Y_t')_{t \geq 0}$ are distributed as continuous-time random walks under $P$. With the usual notation, we define $(X_n) = (Y_{\tau_n})_{n \geq 0}$. By an exponential bound on the event that the Poisson random variable $N_{(u_1 + \eta)N^d}$ does not take a value in $[(u_1 +\eta/2)N^d, \infty)$, we then have
\begin{align*}
{\tilde q} \left[ X_{[0,u_1 N^d]} \cap A \nsubseteq Y_{[0,(u_1+\eta)N^d]} \right] \leq P \left[ N_{(u_1+\eta)N^d} \leq u_1 N^d \right] \leq e^{-c_{u_1,u_2} N^d}.
\end{align*}
Noting that $Y_{[0,(u_1+\eta)N^d]} = Y'_{[\eta N^d, u_2 N^d]}$, \eqref{e:poisson2} of Lemma~\ref{l:poisson} applied with $u=\eta$ and $\epsilon=1/4$ implies that for $N \geq c_{u_1,u_2}$,
\begin{align*}
{\tilde q} \left[Y_{[0,(u_1+\eta)N^d]} \nsubseteq Y'_{[R_2,u_2 N^d]} \right] \stackrel{ \textup{(}2 \leq k^- \textup{)}}{\leq} P \left[R_{k^-} \geq \eta N^d \right] \leq e^{-c_{u_1, u_2} \capacity ({\mathsf A}) }.
\end{align*}
By the estimate \eqref{a:capest} on $\capacity ({\mathsf A})$, the two bounds above complete the proof.
\end{proof}

The lemma above, together with Propositions~\ref{p:c2} and \ref{p:pdom+} now yields the required coupling.

\begin{proposition} \label{p:dom+}
\textup{($d \geq 3$)}
For any $\epsilon \in (0,1), \alpha>0, u>0$, there exists a coupling $Q_3$ of $X_{[0,uN^d]}$ under $P$ with ${\mathcal I}^{u(1+\epsilon)} \cap {\mathsf A}$ under $\mathbb P$, such that for some constant $c=c(u, \epsilon, \alpha)$,
\begin{align}
\label{e:dom+} Q_3 \left[ X(u,{\mathsf A}) \subseteq {\mathcal I}^{u(1+\epsilon)} \cap {\mathsf A} \right] \geq 1 - cN^{-\alpha}.
\end{align}
(recall the definition of $X(u,{\mathsf A})$ in \eqref{e:XuA})
\end{proposition}

\begin{proof}
For $u$ and $\epsilon$ as in the statement, we choose $\epsilon' \in (0, \epsilon/4)$, such that $(1+\epsilon')(1+4 \epsilon') = 1 + \epsilon$ and set $u' = (1+ \epsilon')u$. In particular, we then have $u'(1+4 \epsilon') = u(1+\epsilon)$. We will apply Propositions~\ref{p:c2} and \ref{p:pdom+} with $u$ and $\epsilon$ replaced by $u'$ and $\epsilon'$. Using \eqref{e:dom+4} and \eqref{e:dom00}, as well as Theorem~2.4 on p.~73 in \cite{L85}, there exists a coupling $q$ of ${\mathcal I}^*$ under $Q_2$ and ${\mathcal I}^{u(1+\epsilon)} \cap {\mathsf A}$ under $\mathbb P$, such that 
\begin{align}
\label{dom+1}
q[ \phi({\mathcal I}^*) \subseteq {\mathcal I}^{u(1+\epsilon)} \cap {\mathsf A}]=1. 
\end{align}
We then define $Q_3$, using the couplings $Q_2$ from Proposition~\ref{p:c2} and $\tilde q$ from Lemma~\ref{l:f}, where we set $u_1=u$, $u_2=u'$. For finite sets $S_1 \subseteq {\mathbb T}$, $S_2 \subseteq {\mathsf A}$, we define
\begin{align*}
&Q_3 \left[X_{[0,uN^d]} = S_1, {\mathcal I}^{u(1+ \epsilon)} \cap {\mathsf A} = S_2 \right] \\
&=\sum_{S \subseteq A,S' \subseteq {\mathbb T}} q \left[  {\mathcal I}^{u(1+\epsilon)} \cap {\mathsf A} = S_2  \Big| {\mathcal I}^* = S \right] Q_2 \left[ {\mathcal I}^* = S \Big| Y_{[R_2,u' N^d]} = S'   \right] \\
&\qquad \qquad \times {\tilde q} \left[ Y_{[R_2,u' N^d]}'  = S', X_{[0,uN^d]} = S_1 \right].
\end{align*}
See Figure~\ref{f:Q3} below for an illustration of the coupling $Q_3$.

\begin{figure}[h]
\begin{equation*}
({\mathcal{I}^{u(1+\epsilon)}},{\mathbb{P}}) \quad \overset{q}{\longleftrightarrow} \quad
({\mathcal{I}^*}, Q_2) \quad \overset{Q_2}{\longleftrightarrow} \quad
({Y_{[R_2,u'N^d]}},P) \quad \overset{\tilde q}{\longleftrightarrow} \quad
({X_{[0,uN^d]}},P)
\end{equation*}
\caption{An illustration showing how the coupling $Q_3$ is defined.}\label{f:Q3}
\end{figure}

Then we have $Q_3 [X_{[0,uN^d]}=S_1] = {\tilde q} [X_{[0,uN^d]}=S_1] = P [X_{[0,uN^d]}=S_1]$ and $Q_3 [{\mathcal I}^{u(1+ \epsilon)} \cap {\mathsf A}= S_2] = q[{\mathcal I}^{u(1+ \epsilon)} \cap {\mathsf A} = S_2 ] = {\mathbb P} [{\mathcal I}^{u(1+ \epsilon)} \cap {\mathsf A} = S_2]$, as required. Finally, 
\begin{equation}
\label{dom+2}
\begin{split}
&Q_3 \left[ X(u,{\mathsf A}) \nsubseteq {\mathcal I}^{u(1+\epsilon)} \cap {\mathsf A} \right] \leq q \left[ {\mathcal I}^{u(1+ \epsilon)} \cap {\mathsf A} \nsupseteq \phi({\mathcal I}^* ) \right] + Q_2 \left[ {\mathcal I}^* \nsupseteq {\mathcal I}^+_2 \cap A \right] +\\
&+Q_2 \left[ {\mathcal I}^+_2 \nsupseteq  Y_{[R_2,u'N^d]} \cap A \right] + {\tilde q} \left[ Y'_{[R_2,u'N^d]}  \nsupseteq X_{[0,uN^d]}  \right]
\end{split}
\end{equation}
By \eqref{dom+1}, the first probability on the right-hand side equals zero. By \eqref{e:dom+1} and \eqref{e:dom+3}, the second probability is bounded from above by $c_{\epsilon, \alpha} u N^{-\alpha}$. By \eqref{e:c2}, we have
\begin{align*}
Q_2 \left[ {\mathcal I}^+_2 \nsupseteq  Y_{[R_2,u'N^d]} \cap A \right] \leq e^{-c_{u, \epsilon} \log^2 N},
\end{align*}
while according to Lemma~\ref{l:f} and \eqref{a:capest}, the last probability in \eqref{dom+2} is bounded by $e^{-c_{\epsilon, u} N^{1/2}}$. This shows \eqref{e:dom+} and completes the proof.
\end{proof}

%% file: below.tex
\section{Domination by random walk}
\label{s:dom-}

In this section, we prove the other half of Theorem~\ref{t:dom}, domination of random interlacements by the random walk trajectory, and as a result prove Theorem~\ref{t:dom}. As in the previous section, the key ingredient is again a truncation argument, this time applied to the random interlacement. The argument is again due to Sznitman, given in \cite{Szn09d}, Theorem~3.1, and shows the following result similar to Proposition~\ref{p:pdom+} (recall from the beginning of Section~\ref{s:not} that random interlacements are defined on the space $(\Omega, {\mathcal F}, {\mathbb P})$):

\begin{proposition} \label{p:szn}
\textup{($d \geq 3$)}
 For $r_N = 2[N^{1-\epsilon}/8]$ (cf.~\eqref{e:rN}), $u>0$, $\epsilon \in (0,1)$ and $N \geq c(\epsilon, \alpha)$, there exist random subsets ${\mathcal I}^*, {\bar {\mathcal I}}$ of $\mathsf A$, defined under $(\Omega, {\mathcal F}, {\mathbb P})$, such that
\begin{align}
\label{e:szn1} &{\mathcal I}^{u(1-4 \epsilon)} \cap {\mathsf A} = {\mathcal I}^* \cup {\bar {\mathcal I}}, \\
\label{e:szn2} &{\mathcal I}^*, {\bar {\mathcal I}} \text{ are independent under } {\mathbb P},\\
\label{e:szn3} &{\mathbb P}[{\bar {\mathcal I}} \neq \emptyset] \leq c_{ \epsilon, \alpha} u N^{-\alpha},\\
\label{e:szn4} &{\mathcal I}^* \text{ is stochastically dominated by } \phi({\mathcal I}_2^- \cap A).
\end{align}
\end{proposition}

\begin{proof}
As we now explain, the result follows from \cite{Szn09d}, Theorem~3.1 and its proof, applied with $u$ and $u'$ replaced by the above $u(1-3 \epsilon)$ and $u(1-4 \epsilon)$. By (3.4) and (3.5) in \cite{Szn09d}, the random sets ${\mathcal I}^*$ and ${\bar {\mathcal I}}$ constructed in Theorem~3.1 of \cite{Szn09d} satisfy \eqref{e:szn1} and \eqref{e:szn2} above. The estimate \eqref{e:szn3} is proved in \cite{Szn09d} with $\alpha$ replaced by $d-1$ (note that the theorem there applies to ${\mathbb Z}^{d+1}$). If, in the notation of \cite{Szn09d}, one replaces $r=[8/\epsilon]+1$ in (3.11) by $r=[8\alpha/\epsilon]+2$, however, one indeed obtains \eqref{e:szn3} above (cf. (3.17) and (3.36) in \cite{Szn09d}). Rather than $\phi({\mathcal I}_2^- \cap A)$, the statement in \cite{Szn09d} features the truncated random interlacement in the above \eqref{e:szn4}. The truncated random interlacement is defined in \cite{Szn09d}, (3.2), and, by our construction of ${\mathcal I}_2^-$ in Proposition~\ref{p:c2} and (1.31) in \cite{Szn09d}, stochastically dominated by 
\begin{equation*}
  \bigcup_{w \in \supp \mu_2^-}  \phi  \left( \ran  (w_{. \wedge  T_C})  \cap A \right) \subseteq \phi({\mathcal I}_2^- \cap A). 
\qedhere
\end{equation*}
\end{proof}

From the last proposition, we directly obtain the required coupling:

\begin{proposition} \label{p:dom-}
\textup{($d \geq 3$)}
For any $\epsilon \in (0,1), \alpha>0, u>0$, there exists a coupling $Q_4$ of $X_{[0,uN^d]}$ under $P$ with ${\mathcal I}^{u(1-\epsilon)} \cap {\mathsf A}$ under $\mathbb P$, such that for some constant $c=c(u, \epsilon, \alpha)$,
\begin{align}
\label{e:dom-} Q_4 \left[  {\mathcal I}^{u(1-\epsilon)} \cap {\mathsf A} \subseteq  X(u, {\mathsf A})   \right] \geq 1 - cN^{-\alpha}.
\end{align}
\end{proposition}

\begin{proof}
We will prove the statement for the specific choice $r_N = 2[N^{1-\epsilon}/8]$ for the radius of $A$. The statement for a general $r_N \leq N^{1-\epsilon}$ follows immediately by monotonicity upon replacing $\epsilon$ by $\epsilon/2$. For $\epsilon$ as in the statement, we chose $\epsilon' \in (0,\epsilon/4)$ such that $(1-\epsilon')(1-4 \epsilon')=1-\epsilon$, let $u'=(1-\epsilon')u$ and apply Propositions~\ref{p:szn} and \ref{p:c2} with $\epsilon$ and $u$ replaced by $\epsilon'$ and $u'$. By \eqref{e:szn4} and Theorem~2.4 on p.~73 in \cite{L85}, there exists a coupling $q$ of the random set ${\mathcal I}^*$ under $\mathbb P$ and ${\mathcal I}_2^- \cap A$ under $Q_2$, such that
\begin{align}
\label{dom-1}
 q \left[ {\mathcal I}^* \subseteq \phi({\mathcal I}_2^- \cap A) \right] = 1.
\end{align}
For sets $S_1 \subseteq {\mathsf A}$ and $S_2 \subseteq {\mathbb T}$, we then define
\begin{align*}
Q_4 &\left[ {\mathcal I}^{u(1- \epsilon)} \cap {\mathsf A} = S_1, X_{[0,uN^d]} = S_2  \right] = \\
&\sum_{\substack{S, S' \subseteq A \\ S'' \subseteq {\mathsf A}}} P \left[ X_{[0,uN^d]} = S_2 | Y_{[0,u'N^d]} = S \right] Q_2 \left[ Y_{[0,u'N^d]}  = S \Big|  {\mathcal I}_2^- \cap A = S' \right]\\ 
& \qquad \times q \left[ {\mathcal I}_2^- \cap A = S' \Big| {\mathcal I}^* = S'' \right] {\mathbb P} \left[ {\mathcal I}^{u'(1-4\epsilon')} \cap {\mathsf A} = S_1, {\mathcal I}^* = S'' \right]. 
\end{align*}
Then we have $Q_4 \left[{\mathcal I}^{u(1-\epsilon)} \cap {\mathsf A} = S_1 \right] = {\mathbb P}  \left[{\mathcal I}^{u(1-\epsilon)} \cap {\mathsf A} = S_1 \right]$  and $Q_4 \left[ X_{[0,uN^d]} = S_2 \right] = P \left[ X_{[0,uN^d]} = S_2 \right],$ as required for a coupling.
See Figure~\ref{f:Q4} below for an illustration of the coupling $Q_4$.

\begin{figure}[h]
\begin{equation*}
({\mathcal{I}^*},{\mathbb{P}}) \quad \overset{q}{\longleftrightarrow} \quad
({\mathcal{I}_2^- \cap A}, Q_2) \quad \overset{Q_2}{\longleftrightarrow} \quad
({X_{[0,uN^d]}},P)
\end{equation*}
\caption{An illustration showing how the coupling $Q_4$ is defined.}\label{f:Q4}
\end{figure}
Finally,
\begin{align*}
 &Q_4 \left[  {\mathcal I}^{u(1-\epsilon)} \cap {\mathsf A} \nsubseteq  X(u,{\mathsf A}) \right] \leq P \left[ X_{[0,uN^d]} \cap A \nsupseteq Y_{[0,u'N^d]} \cap A \right]\\
&\qquad + Q_2 \left[ Y_{[0,u'N^d]} \cap A \nsupseteq {\mathcal I}_2^- \cap A \right] + q \left[ \phi( {\mathcal I}_2^- \cap A) \nsupseteq {\mathcal I}^* \right] + {\mathbb P} \left[ {\mathcal I}^* \nsupseteq {\mathcal I}^{u'(1-4\epsilon')} \cap {\mathsf A} \right]\\
& \qquad \leq e^{-c_{u, \epsilon} N^d} + e^{-c_{u, \epsilon} \log^2 N} + c_{u, \epsilon, \alpha} N^{-\alpha},
\end{align*}
where we have used an exponential bound on the probability that the Poisson random variable $N_{u'N^d}$ takes a value larger than $uN^d$, \eqref{e:c2}, \eqref{dom-1}, \eqref{e:szn1} and \eqref{e:szn3} for the final estimate. Hence, \eqref{e:dom-} holds and this completes the proof of Proposition~\ref{p:dom-}.
\end{proof}

Finally, we can deduce Theorem~\ref{t:dom} from Propositions~\ref{p:dom+} and \ref{p:dom-}.

\begin{proof}[Proof of Theorem~\ref{t:dom}.]
The statement follows immediately from Propositions~\ref{p:dom+} and \ref{p:dom-}: For any sets $S_1 \subseteq {\mathsf A}$, $S_2 \subseteq {\mathbb T}$ and $S_3 \subseteq {\mathsf A}$, we define
\begin{align*}
 &\coup \left[ {\mathcal I}^{u(1-\epsilon)} \cap {\mathsf A} = S_1, X_{[0,uN^d]} =S_2, {\mathcal I}^{u(1+\epsilon)} \cap {\mathsf A} = S_3 \right] \\
& = Q_3 \left[ {\mathcal I}^{u(1-\epsilon)} \cap {\mathsf A} = S_1 \Big| X_{[0,uN^d]} = S_2 \right] Q_4 \left[ X_{[0,uN^d]} = S_2 , {\mathcal I}^{u(1+\epsilon)} \cap {\mathsf A} = S_3 \right],
\end{align*}
where the right-hand side is understood to be $0$ if $P[X_{[0,uN^d]}=S_2]=0$. As with the previous couplings, one checks that $\coup$ has the required properties.
\end{proof}

%% file: torus.tex
\appendix
\section*{Appendix}
\label{s:appendix}

\renewcommand{\theequation}{A.\arabic{equation}}
\renewcommand{\thetheorem}{A.\arabic{theorem}}
\setcounter{theorem}{0}

\begin{proof}[Proof of Lemma~\ref{l:basic}]
 The bound \eqref{a:gotoC} follows from \cite{L91}, Proposition 1.5.10, p.~36. In order to prove \eqref{a:Drat}, 
we write $A \cup \partial_e A = {\bar A}$. Using the canonical projection $\Pi$ from ${\mathbb Z}^d$ onto $\mathbb T$, we can bound $P_x \left[ H_{\bar A} \leq \reg \right]$ by 
\begin{align}
\label{basic1}
P_{\phi(x)} \left[ T_{B(\phi(x),N \log^2 N)} \leq \reg \right] + P_{\phi(x)} \left[ H_{\Pi^{-1} ({\bar A}) \cap B(\phi(x),N \log^2 N)} < \infty \right].
\end{align}
By Fubini's theorem and Azuma's inequality (cf.~\cite{AS92}, p.~85),
\begin{align*}
 P_{\phi(x)} \left[ T_{B(\phi(x),N \log^2 N)} \leq \reg \right] &\leq E \left[ P_{\phi(x)} \left[ |X_k| \geq N \log^2 N \text{ for some } k \leq n \right] \Big|_{n=N_\reg} \right]\\
&\leq cE \left[ \exp \left( - c (N \log^2 N)^2 /N_\reg \right) \right].
\end{align*} 
With a bound of $e^{-c \reg}$ on the probability that the Poisson random variable $N_\reg$ is larger than $2 \reg$, we deduce that
\begin{align*}
 P_{\phi(x)} \left[ T_{B(\phi(x),N \log^2 N)} \leq \reg \right] &\leq c  e^{-c \log^2 N}.
\end{align*}
The set $\Pi^{-1} ({\bar A}) \cap B(\phi(x),N \log^2 N)$ is contained in a union of the ball $B(0,N^{1-\epsilon})$ and no more than $\log^c N$ translated copies of it. By choice of $x$, $\phi(x)$ is at distance at least $c N^{1-\epsilon/2}$ from each ball. Hence, using the union bound and again the estimate in \cite{L91}, Proposition 1.5.10 on the hitting probability, we obtain
\begin{align*}
 P_{\phi(x)} \left[ H_{\Pi^{-1} ({\bar A}) \cap B(\phi(x),N \log^2 N)} < \infty \right] \leq c_\epsilon (\log N)^c N^{-c_\epsilon}.
\end{align*}
Inserting the last two estimates into \eqref{basic1}, we have shown \eqref{a:Drat}. The proof of \eqref{a:Drat'} is analogous.
\end{proof}

\begin{proof}[Proof of Lemma~\ref{l:quasi}]
Parts of the proof are contained in \cite{K79}. Since ${\mathbb T} \setminus B$ is connected, the following statement holds (see \cite{K79}, page~91, equation (6.6.3) for a proof):

\begin{lemma}
\label{l:qconv}
\textup{($d \geq 2$)}
For any vertices $x_0, x \in {\mathbb T} \setminus B$ and fixed $N \geq 1$,
\begin{align}
\lim_{t \to \infty} P_{x_0} [Y_t=x|H_B >t] = \sigma(x).
\end{align}
\end{lemma}

The above lemma applies for fixed $N$, but we require an estimate for all $N$ and $t_N = \reg$. To this end, we need the following lower bound on the quasistationary distribution.

\begin{lemma}
\label{l:qell}
\textup{($d \geq 3$)}
\begin{align}
\label{e:qell}
\inf_{x \in {\mathbb T} \setminus B} \sigma(x) \geq \frac{c_\epsilon}{N^{2d}}.
\end{align}
\end{lemma}

\begin{proof}[Proof of Lemma~\ref{l:qell}.]
Let $x \in {\mathbb T} \setminus B$ and choose $x' \in {\mathbb T} \setminus B$ such that $\sigma(x') \geq 1/(N^d-|B|) \geq c_\epsilon/N^d$. By reversibility, we have, for $t>0$, 
\begin{equation} 
\label{qell1}
\begin{split}
P_{x'} [Y_t=x|H_B >t] = P_x [Y_t=x'|H_B>t] \frac{P_x[H_B>t]}{P_{x'}[H_B>t]}.
\end{split}
\end{equation}
In order to find a lower bound on the fraction, observe that
\begin{align}
\label{qell2}
P_x[H_B>t] &\geq P_x[H_{x'} < H_B, H_B \circ \theta_{H_{x'}} >t] = P_x [H_{x'} < H_B] P_{x'} [H_B >t]. 
\end{align}
We now want a lower bound on $P_x [H_{x'} < H_B]$.  For any $z \in {\mathbb T} \setminus B(0,N/4)$, the Harnack inequality (cf.~\cite{L91}, Theorem~1.7.1, p.~42), applied to the harmonic function $y \in B^c \mapsto P_y [H_z < H_B]$, together with a standard covering argument, shows that $P_y [H_z < H_B] \geq c_\epsilon \inf_{y' \in B(z,N/10)} P_{y'} [H_z < H_B]$ for any $y,z \in {\mathbb T} \setminus B(0,N/4)$. In particular, using \cite{L91}, Proposition~1.5.10, p.~36, to bound the hitting probability from below, we have
\begin{align}
\label{qell3}
 \inf_{y,z \in {\mathbb T} \setminus B(0,N/4)} P_y [H_z < H_B] \geq c_\epsilon N^{2-d}.
\end{align}
In addition, an elementary estimate on one-dimensional simple random walk shows that $P_x [T_{B(0,N/4)} < H_B] \geq c/N$, and analogously for $x'$. With the strong Markov property applied at time $T_{B(0,N/4)}$, we find that
\begin{align*}
 P_x [H_{x'} < H_B]  &\geq P_x [T_{B(0,N/4)} < H_B] \negmedspace \inf_{y \in B(0,N/4)^c}\negmedspace  P_y [H_{x'} < H_B] \geq \frac{c}{N} \inf_{y \in B(0,N/4)^c} \negmedspace  P_{x'} [H_y < H_B],
\end{align*}
using reversibility to exchange the roles of $x'$ and $y$. By \eqref{qell3} and again the strong Markov property at time $T_{B(0,N/4)}$, we find that the last probability on the right-hand side is bounded from below by $\frac{c}{N} \times c_\epsilon N^{2-d}$. Inserting into \eqref{qell2}, have $P_x[H_B>t] \geq c_\epsilon N^{-d} P_{x'}[H_B>t]$, from which we infer with \eqref{qell1} that for all $t \geq 0$,
\begin{align*}
P_{x'} [Y_t=x|H_B >t] \geq P_x [Y_t=x'|H_B>t] c_\epsilon N^{-d}.
\end{align*}
By Lemma~\ref{l:qconv}, the two sides in this inequality converge as $t \to \infty$ to $\sigma(x) \geq \sigma(x') c_\epsilon N^{-d}$,
and $x'$ was chosen such that $\sigma(x') \geq c_\epsilon /N^d$. This completes the proof of Lemma~\ref{l:qell}.
\end{proof}

Recall that $\lambda^B_1$ denotes the largest eigenvalue of $P^B$. Let $\lambda^B_2$ be the second largest eigenvalue of $P^B$ (cf.~\eqref{e:transB}). The next lemma shows that the spectral gap of $P^B$ is of at least the same order $cN^{-2}$ as the spectral gap of $P$ itself.

\begin{lemma}
\label{l:AB}
\textup{($d \geq 3$)}
\begin{align}
\label{e:AB2}
&\lambda^B_1 - \lambda^B_2 \geq c_\epsilon N^{-2}.
\end{align}
\end{lemma}

\begin{proof}[Proof of Lemma~\ref{l:AB}.]
Consider the complete transition matrix $((2d)^{-1} \mathbf{1}_{x \sim y})_{x,y \in {\mathbb T}}$ and let $\lambda_2$ be its second largest eigenvalue. By the eigenvalue interlacing inequality (cf.~\cite{H95}, Corollary~2.2), we have
$\lambda^B_2 \leq \lambda_2,$
while by Aldous and Brown \cite{AB}, Lemma~2, and the paragraph following equation~(12),
\begin{align*}
\lambda^B_1 = 1- \frac{1}{E_\sigma[H_B]} \geq 1- \frac{1}{E[H_B]} \stackrel{\eqref{e:Hlbd}, \eqref{a:capestB}}{\geq} 1-c_\epsilon N^{-2- \epsilon (d-2)/2},
\end{align*}
hence, using that $1-\lambda_2 \geq cN^{-2}$ (cf.~Remark~2.2 in \cite{W08}), $\lambda^B_1 - \lambda^B_2 \geq 1 - \lambda_2 - c_\epsilon N^{-2 - \epsilon (d-2) /2} \geq c_\epsilon N^{-2},$ proving Lemma~\ref{l:AB}.
\end{proof}

Using the restricted transition matrix $P^B$ defined in \eqref{e:transB}, the conditional probability in \eqref{e:quasi} is given by
\begin{align}
\label{quasi3}
P_x[Y_\reg=y|H_B>\reg] = \frac{ \delta_x^T e^{-\reg(I-P^B)} \delta_y}{\delta_x^T e^{-\reg (I-P^B)} \mathbf{1}},
\end{align}
where, for $x \in {\mathbb T} \setminus B$, $\delta_x$ denotes the vector with $x$-entry $1$ and all other entries $0$, and $\mathbf 1$ denotes the vector with all entries equal to $1$. Let now $m=N^d-|B|$, and let $\lambda^B_1 \geq \lambda^B_2 \geq \cdots \geq \lambda^B_m$ be the eigenvalues of $P^B$ in decreasing order with orthonormal eigenvectors $v_1, \ldots, v_m$. As in \cite{K79}, we now introduce the matrices $J$ and $\Delta$, 
$J=v_1 v_1^T, \, \Delta = P^B- \lambda^B_1 J.$
It is then elementary to check that $\Delta J = J \Delta = 0$ and that $J^2=J$. Hence, we have
\begin{equation}
\label{quasi3.1}
\begin{split}
e^{-\reg(I-P^B)} &= e^{-\reg I} \bigg( I+ \sum_{k \geq 1} \frac{\reg^k}{k \textrm{!}} \left( \Delta^k + (\lambda^B_1)^k J \right) \bigg) \\
&= e^{-\reg I} \left( e^{\reg \Delta} + e^{\reg \lambda^B_1} J - J \right) = e^{-\reg(I-\Delta)} + e^{-\reg(1-\lambda^B_1)} J  - e^{-\reg}J.
\end{split}
\end{equation}
Let us now write $\delta_y = \sum_{i=1}^m a_i v_i, \text{ where } a_i = v_i^T \delta_y.$
Since $\Delta v_i = \mathbf{1}_{i\neq 1} \lambda_i v_i$, \eqref{quasi3.1} implies that $e^{-\reg (I-P^B)} \delta_y$ equals
\begin{equation}
\label{quasi4}
\begin{split}
e^{-\reg(1-\lambda^B_1)} &\bigg(a_1 e^{- \lambda^B_1 \reg} v_1 + \sum_{i=2}^m a_i e^{-(\lambda^B_1-\lambda^B_i)\reg} v_i - e^{- \lambda^B_1 \reg} J \delta_y + J \delta_y  \bigg)\\
&= e^{-\reg(1-\lambda^B_1)} \Big( J \delta_y  + \sum_{i=2}^m a_i e^{-(\lambda^B_1-\lambda^B_i)\reg} v_i \Big) =e^{- \reg(1-\lambda^B_1)} \left( J \delta_y + \phi_N \right),
\end{split}
\end{equation} 
where $\phi_N$ is defined by this last equation, and by Pythagoras' theorem, \eqref{e:reg} and \eqref{e:AB2}, has $\ell^2$-norm bounded by $|\phi_N|_2 \leq e^{- \reg (\lambda^B_1 - \lambda^B_2)} \leq e^{-c_\epsilon \log^2 N}$. Similarly, we have
\begin{align}
\label{quasi5}
e^{-\reg(I-P^B)} \mathbf{1} &= e^{- \reg (1-\lambda^B_1)} \left( J \mathbf{1} + \phi_N' \right),
\end{align}
where $|\phi_N'|_2 \leq  e^{-c_\epsilon \log^2 N}.$ We have $\delta_x^T J \delta_y = (v_1)_x (v_1)_y$ and $\delta_x^T J \mathbf{1} = (v_1)_x v_1^T \mathbf{1}$. In particular, by \eqref{d:quasi} and \eqref{e:qell}, both $\delta_x^T J \delta_y$ and $\delta_x^T J \mathbf{1}$ are bounded from below by $c_\epsilon N^{-4d} \gg e^{-c_\epsilon \log^2 N}$. Inserting \eqref{quasi4} and \eqref{quasi5} into \eqref{quasi3}, we hence obtain that
\begin{align*}
P_x[Y_\reg=y|H_B>\reg]= \frac{(v_1)_y}{v_1^T \mathbf{1}} + \phi^{''}_N \stackrel{\eqref{d:quasi}}{=} \sigma(y) + \phi^{''}_N ,
\end{align*}
where $|\phi^{''}_N|$ again satisfies $|\phi_N^{''}| \leq  e^{-c_\epsilon \log^2 N}$. This
completes the proof of Lemma~\ref{l:quasi}.
\end{proof}

%% file: domination.bbl
\def\cprime{$'$} \def\cprime{$'$}
\begin{thebibliography}{10}

\bibitem{AB}
David~J. Aldous and Mark Brown.
\newblock Inequalities for rare events in time-reversible {M}arkov chains. {I}.
\newblock In {\em Stochastic inequalities ({S}eattle, {WA}, 1991)}, volume~22
  of {\em IMS Lecture Notes Monogr. Ser.}, pages 1--16. Inst. Math. Statist.,
  Hayward, CA, 1992.

\bibitem{AS92}
Noga Alon and Joel~H. Spencer.
\newblock {\em The probabilistic method}.
\newblock Wiley-Interscience Series in Discrete Mathematics and Optimization.
  John Wiley \& Sons Inc., New York, 1992.
\newblock With an appendix by Paul Erd{\H{o}}s, A Wiley-Interscience
  Publication.

\bibitem{BS08}
Itai Benjamini and Alain-Sol Sznitman.
\newblock Giant component and vacant set for random walk on a discrete torus.
\newblock {\em J. Eur. Math. Soc. (JEMS)}, 10(1):133--172, 2008.

\bibitem{CTW}
Ji{\v r{\'i}} {\v C}ern{\'y}, Augusto Teixeira, and David Windisch.
\newblock Giant vacant component left by a random walk in a random d-regular
  graph.
\newblock preprint available at \url{http://www.math.ethz.ch/\textasciitilde
  cerny/}, 2009.

\bibitem{CF10}
Colin Cooper and Alan Frieze.
\newblock Component structure induced by a random walk on a random graph.
\newblock preprint available at \url{http://arxiv.org/abs/1005.1564}, 2010.

\bibitem{D05}
Richard Durrett.
\newblock {\em Probability: {T}heory and {E}xamples}.
\newblock Duxbury Press, Belmont, CA, third edition, 2005.

\bibitem{FP99}
P.~J. Fitzsimmons and Jim Pitman.
\newblock Kac's moment formula and the {F}eynman-{K}ac formula for additive
  functionals of a {M}arkov process.
\newblock {\em Stochastic Process. Appl.}, 79(1):117--134, 1999.

\bibitem{H95}
Willem~H. Haemers.
\newblock Interlacing eigenvalues and graphs.
\newblock {\em Linear Algebra Appl.}, 226/228:593--616, 1995.

\bibitem{K79}
Julian Keilson.
\newblock {\em Markov chain models---rarity and exponentiality}, volume~28 of
  {\em Applied Mathematical Sciences}.
\newblock Springer-Verlag, New York, 1979.

\bibitem{K81}
H.~Kesten.
\newblock Percolation theory for mathematicians.
\newblock {\em Nieuw Arch. Wisk. (3)}, 29(3):227--239, 1981.

\bibitem{K59}
R.~Z. Khas{\cprime}minski{\u\i}.
\newblock On positive solutions of the equation {${\mathcal U}+Vu=0$}.
\newblock {\em Theor. Probability Appl.}, 4:309--318, 1959.

\bibitem{L91}
Gregory~F. Lawler.
\newblock {\em Intersections of random walks}.
\newblock Probability and its Applications. Birkh{\"a}user Boston Inc., Boston,
  MA, 1991.

\bibitem{LPW09}
David~A. Levin, Yuval Peres, and Elizabeth~L. Wilmer.
\newblock {\em Markov chains and mixing times}.
\newblock American Mathematical Society, Providence, RI, 2009.
\newblock With a chapter by James G. Propp and David B. Wilson.

\bibitem{L85}
Thomas~M. Liggett.
\newblock {\em Interacting particle systems}, volume 276 of {\em Grundlehren
  der Mathematischen Wissenschaften [Fundamental Principles of Mathematical
  Sciences]}.
\newblock Springer-Verlag, New York, 1985.

\bibitem{SC97}
Laurent Saloff-Coste.
\newblock Lectures on finite {M}arkov chains.
\newblock In {\em Lectures on probability theory and statistics
  ({S}aint-{F}lour, 1996)}, volume 1665 of {\em Lecture Notes in Math.}, pages
  301--413. Springer, Berlin, 1997.

\bibitem{S02}
Denis Serre.
\newblock {\em Matrices}, volume 216 of {\em Graduate Texts in Mathematics}.
\newblock Springer-Verlag, New York, 2002.
\newblock Theory and applications, Translated from the 2001 French original.

\bibitem{SS10}
Vladas Sidoravicius and Alain-Sol Sznitman.
\newblock Connectivity bounds for the vacant set of random interlacements.
\newblock to appear in Ann. Inst. H. Poincar{\'e}.

\bibitem{SS09}
Vladas Sidoravicius and Alain-Sol Sznitman.
\newblock Percolation for the vacant set of random interlacements.
\newblock {\em Comm. Pure Appl. Math.}, 62(6):831--858, 2009.

\bibitem{Szn09c}
Alain-Sol Sznitman.
\newblock On the domination of random walk on a discrete cylinder by random
  interlacements.
\newblock {\em Electron. J. Probab.}, 14:no. 56, 1670--1704, 2009.

\bibitem{Szn09d}
Alain-Sol Sznitman.
\newblock Upper bound on the disconnection time of discrete cylinders and
  random interlacements.
\newblock {\em Ann. Probab.}, 37(5):1715--1746, 2009.

\bibitem{Szn09}
Alain-Sol Sznitman.
\newblock Vacant set of random interlacements and percolation.
\newblock Ann. of Math. (2), 171, No. 3, 2039-2087, 2010.

\bibitem{T08}
Augusto Teixeira.
\newblock On the uniqueness of the infinite cluster of the vacant set of random
  interlacements.
\newblock {\em Ann. Appl. Probab.}, 19(1):454--466, 2009.

\bibitem{T09b}
Augusto Teixeira.
\newblock On the size of a finite vacant cluster of random interlacements with
  small intensity.
\newblock arXiv:1002.4995, to appear in Probability Theory and Related Fields,
  2010.

\bibitem{W08}
David Windisch.
\newblock Random walk on a discrete torus and random interlacements.
\newblock {\em Electron. Commun. Probab.}, 13:140--150, 2008.

\end{thebibliography}
